\documentclass[11pt]{amsart}
\usepackage{amsmath, amssymb, amscd, amsthm, amsfonts}
\usepackage{mathtools}
\usepackage{graphicx}

\usepackage{todonotes}
\usepackage{xcolor}

\usepackage{enumitem}
\usepackage[hidelinks]{hyperref}
\setlist[enumerate]{itemsep=2pt,parsep=2pt,before={\parskip=2pt}}

\usepackage[backend=bibtex,giveninits=true,maxbibnames=99, doi=false,isbn=false,url=false,eprint=false]{biblatex}

\bibliography{refs}

\renewbibmacro{in:}{}

\newcommand{\cosimp}[3]{\xymatrix@1{#1 \ar@<.4ex>[r] \ar@<-.4ex>[r] & {\ }#2 \ar@<0.8ex>[r] \ar[r] \ar@<-.8ex>[r] & {\ } #3 \ar@<1.2ex>[r] \ar@<.4ex>[r] \ar@<-.4ex>[r] \ar@<-1.2ex>[r] & \cdots }}

\newcommand{\adjunction}[4]{\xymatrix@1{#1{\ } \ar@<0.3ex>[r]^{ {\scriptstyle #2}} & {\ } #3 \ar@<0.3ex>[l]^{ {\scriptstyle #4}}}}

\numberwithin{equation}{section}

\newtheorem{theorem}{Theorem}[section]
\newtheorem{proposition}[theorem]{Proposition}
\newtheorem{lemma}[theorem]{Lemma}
\newtheorem{corollary}[theorem]{Corollary}
\newtheorem{question}[theorem]{Question}
\theoremstyle{definition}
\newtheorem{remark}[theorem]{Remark}
\newtheorem{example}[theorem]{Example}
\newtheorem{definition}[theorem]{Definition}
\newtheorem{claim}[theorem]{Claim}
\newtheoremstyle{customNumber}
     {}          
     {}          
     {\itshape}  
     {}          
     {\bfseries} 
     {.}         
     { }         
     {\thmname{#1}\thmnumber{ #2}\thmnote{ #3}}
\theoremstyle{customNumber}

\renewcommand{\phi}{\varphi}

\newcommand{\abs}[1]{\lvert#1\rvert}
\newcommand{\CAT}{\operatorname{CAT}}
\newcommand{\diam}{\operatorname{diam}}

\DeclareMathOperator{\N}{\mathbb{N}}
\DeclareMathOperator{\R}{\mathbb{R}}
\DeclareMathOperator{\Z}{\mathbb{Z}}

\DeclareMathOperator{\Lip}{Lip}

\DeclareMathOperator{\sing}{sing}

\DeclareMathOperator{\vertices}{vert}
\DeclareMathOperator{\edges}{edges}

\DeclareMathOperator{\Star}{star}
\DeclareMathOperator{\sep}{sep}

\DeclareMathOperator{\rad}{rad}
\DeclareMathOperator{\Conn}{Conn}
\DeclareMathOperator{\up}{Up}
\DeclareMathOperator{\GoodName}{\text{weighted tree}}

\DeclareMathOperator{\dist}{dist}
\DeclareMathOperator{\id}{id}

\DeclareMathOperator{\bl}{bilip}
\DeclareMathOperator{\pred}{pred}

\begin{document}

\title[Approximating spaces of Nagata dimension zero by weighted trees]%
{Approximating spaces of Nagata dimension zero by \\ weighted trees}

\author{Giuliano Basso}
\address{Max Planck Institute for Mathematics,
Vivatsgasse 7,
53111 Bonn,
Germany \& Department of Mathematics\\ University of Fribourg\\ Chemin du Mus\'ee 23\\ 1700 Fribourg \\ Switzerland}
\email{basso@mpim-bonn.mpg.de}

\author{Hubert Sidler}
\address{Department of Mathematics\\ University of Fribourg\\ Chemin du Mus\'ee 23\\ 1700 Fribourg \\ Switzerland}
\email{hubert.sidler@unifr.ch}

\keywords{Nagata dimension, weighted tree, ultrametric space, Lipschitz extension}

\subjclass[2020]{Primary 30L05; Secondary 54F45 and 54C20}

\maketitle

\begin{abstract}
We prove that if a metric space \(X\) has Nagata dimension zero with constant \(c\), then there exists a dense subset of \(X\) that is \(8c\)-bilipschitz equivalent to a weighted tree. The factor \(8\) is the best possible if \(c=1\), that is, if \(X\) is an ultrametric space. This yields a new proof of a result of Chan, Xia, Konjevod and Richa. Moreover, as an application, we also obtain quantitative versions of certain metric embedding and Lipschitz extension results of Lang and Schlichenmaier. Finally, we prove a variant of our main theorem for  \(0\)-hyperbolic proper metric spaces. This generalizes a result of Gupta. 
\end{abstract}

\section{Introduction}

\subsection{Statement of main result}

In this article we study how good a given metric space may be approximated by weighted trees. For us, a \textit{weighted tree} is a metric space \((V, d)\) that is obtained by equipping a graph-theoretic weighted tree \(T=(V, E, \omega)\) with its shortest-path metric \(d\). Bilipschitz approximability by weighted trees is a useful notion as weighted trees have many desirable metric embedding and Lipschitz extension properties (see e.g. \cite{MR1765236}, \cite{MR1337355}, \cite{MR1056175}, \cite{MR1738681}). For example, Matou\v{s}ek \cite{MR1738681} proved that any \(n\)-point weighted tree can be embedded into \(\ell_p\), \(1<p<\infty\), with distortion at most \(C_p \cdot (\log \log n)^{\min\{ 1/p, 1/2\}}\). Hence, if every member of a given class of metric spaces is \(C\)-bilipschitz equivalent to some \(n\)-point weighted tree, this directly leads to a corresponding embeddability result for such metric spaces. 

The main result of the present article is the following bilipschitz approximation result for spaces of Nagata dimension zero. This includes  finite metric spaces and ultrametric spaces.

\begin{theorem}\label{thm:main1}  
If a metric space \(X\) has Nagata dimension zero with constant \(c\),  then there exists a dense subset of \(X\) that is \(8c\)-bilipschitz equivalent to a weighted tree. If \(c=1\), then the factor \(8\) is the best possible.
\end{theorem}

Throughout this article, we call a map \(f\colon X\to Y\) between metric spaces $C$-\emph{bilipschitz} for some constant $C\geq 0$ if there exists $s>0$ such that \(sd(x,y) \leq d(f(x),f(y)) \leq C sd(x,y)\) for all \(x,y\in X\). We call two metric spaces \(C\)-\emph{bilipschitz equivalent}, if there exists a bijective and \(C\)-bilipschitz map between them.

For the special case that \(X\) is a finite \(0\)-hyperbolic metric space, Gupta \cite{MR1958411} has proved a result similar to Theorem~\ref{thm:main1}; see Theorem~\ref{thm:gupta} below for the exact statement. 
We remark that for uncountable spaces \(X\), one cannot expect that \(X\) is bilipschitz equivalent to any weighted tree. 
Hence, the conclusion of Theorem~\ref{thm:main1} cannot be strengthened in general. This can be seen by considering the \(p\)-adic integers (see Example~\ref{ex:2-adic-numbers}).

Roughly speaking, Nagata dimension is a variant of Gromov's asymptotic dimension taking into account local and global properties of the metric space (see \cite{MR2423966}, \cite{MR1253544}). 
The following definition is due to Assouad (see \cite{MR651069}). We remark that Assouad's work is closely related to earlier work of Nagata \cite{MR105081} on the dimension theory of metric spaces.

\begin{definition}
A metric space \(X\) has \textit{Nagata dimension zero with constant} \(c\) if for all real numbers \(s>0\), \(X\) admits a covering \(\mathcal{B}_s\) with \(\diam(B)\leq cs\) for every \(B\in \mathcal{B}_s\) and \(d(B, B')> s\) for all distinct \(B\), \(B'\in \mathcal{B}_s\). 
\end{definition}

Clearly, every finite metric space has Nagata dimension zero. Hence, Theorem~\ref{thm:main1} gives a computable upper bound on how good a given finite metric space may be approximated by a weighted tree. Furthermore, any ultrametric space has Nagata dimension zero with constant one. Hence, Theorem~\ref{thm:main1} can also be applied to ultrametric spaces. It turns out that for such spaces the constant \(8\) appearing in Theorem~\ref{thm:main1} is the best possible. This follows directly from a remarkable result of Chan, Xia, Konjevod and Richa \cite{MR2304999}; see Theorem~\ref{thm:lowerBound} below. We suspect that Theorem~\ref{thm:main1} is also optimal when \(c>1\). 

There are several equivalent characterizations of spaces of Nagata dimension zero. For instance, 
a metric space has Nagata dimension zero if and only if it is uniformly disconnected in the sense of David and Semmes (see \cite[Definition 15.1]{MR1616732}), or if and only if it has Lipschitz dimension zero in the sense of Cheeger and Kleiner (see \cite[Proposition 6.3]{MR4205978}). Nagata dimension has been studied extensively by many authors (see e.g.  \cite{MR2520117}, \cite{MR2340955}, \cite{MR2418302}, \cite{MR2200122}, \cite{MR3320519}, \cite{MR2372759}). Moreover, it has been used in various contexts such as in  \cite{MR3881831}, \cite{MR3641484}, \cite{MR3268779}. For recent results concerning the Nagata dimension of minor-closed families of graphs we refer to the articles \cite{bonamy2020asymptotic}, \cite{fujiwara2020asymptotic},  \cite{jorgensen2020geodesic}, \cite{MR3349105}.

Theorem \ref{thm:main1} enables us to improve some embedding and extension results of Lang and Schlichenmaier \cite{MR2200122}. We discuss this in detail in the next paragraph. See Corollaries~\ref{cor:biLip} and \ref{cor:main2} below.

\subsection{Applications} Lang and Schlichenmaier proved that whenever a metric space has Nagata dimension at most \(n\), then it admits a quasisymmetric embedding into the product of \(n+1\) many \(\R\)-trees (see \cite[Corollary 1.4]{MR2200122}).  As a direct consequence of Theorem~\ref{thm:main1}, we have the following improvement of their result in the case when \(n=0\).

\begin{corollary}\label{cor:biLip}
Suppose \(X\) is a complete metric space which has Nagata dimension zero with constant \(c\). Then there exists an \(8c\)-bilipschitz embedding \(f\colon X\to Y\) into a complete \(\R\)-tree \(Y\) such that \(\sing Y\subset f(X)\). 
\end{corollary}

An \textit{\(\R\)-tree} is a uniquely geodesic metric space \(Y\) such that whenever \([y_1, y_2] \cap [y_2,y_3] = \{y_2\}\) then one has \([y_1, y_2] \cup [y_2,y_3] = [y_1,y_3]\) for all \(y_1\), \(y_2\), \(y_3\in Y\). Here, we use \([y_i,y_j]\) to denote the unique geodesic in \(Y\) connecting \(y_i\) to \(y_j\).
We say that \(y\in Y\) is \textit{regular} if there exists an isometric embedding \(U(y, \varepsilon) \to \R\) for some \(\varepsilon>0\), where \(U(y, \varepsilon)\) denotes the open ball of radius \(\varepsilon\) centered at \(y\). Every point which is not regular is said to be \textit{singular}, and we denote by \(\sing Y\) the set of singular points of \(Y\). Notice that if the \(\R\)-tree \(Y=\abs{T}\) is a geometric realization of a weighted tree \(T=(V, E, \omega)\), then \(\sing Y\) consists precisely of those vertices of \(T\) with degree distinct from two. 

A metric space has Nagata dimension zero if and only if it is bilipschitz equivalent to an ultrametric space (see, for example, \cite[Theorem 3.3]{MR2340955} or \cite[Proposition 15.7]{MR1616732}), and every ultrametric space admits an isometric embedding into an \(\R\)-tree (every ultrametric space is $0$-hyperbolic, so the claim follows from Proposition~\ref{prop:dress-1}). Consequently, any space of Nagata dimension zero admits a bilipschitz embedding \(f\) into an \(\R\)-tree \(Y\). Hence, the new part of Corollary~\ref{cor:biLip} is that there is such an \(f\) which satisfies \(\sing Y \subset f(X)\). 

As a direct consequence of the inclusion \(\sing Y \subset f(X)\), it is possible to derive a Lipschitz extension result. Indeed, if \(Y=\abs{T}\) for some weighted tree \(T\), then by interpolating linearly between the images of adjacent vertices, any \(1\)-Lipschitz map defined on \(\vertices T\) to a Banach space admits a \(1\)-Lipschitz extension to \(\abs{T}\). This observation generalizes to \(1\)-Lipschitz maps defined on \(\sing Y\) for all complete \(\R\)-trees \(Y\).

\begin{corollary}\label{cor:main2}
Let \(X\), \(Y\) be complete metric spaces and \(Z\subset X\) a closed subset which has Nagata dimension zero with constant \(c\). Suppose that \(Y\) is \(\lambda\)-quasiconvex.
Then every \(1\)-Lipschitz map \(f\colon Z\to Y\) admits an \(8c \lambda\)-Lipschitz extension \(\bar{f}\colon X\to Y\). 
\end{corollary}
Here, a metric space \(Y\) is called \(\lambda\)-\textit{quasiconvex}, for some \(\lambda\geq 1\), if all \(y_1\), \(y_2\in Y\) can be connected via a continuous curve \(\gamma\colon [a,b]\to X\) such that \(\ell(\gamma)\leq \lambda d(y_1,y_2)\), where \(\ell(\gamma)\) denotes the length of \(\gamma\). 
Corollary \ref{cor:main2} is an explicit version of a Lipschitz extension result due to Lang and Schlichenmaier (see \cite[Theorem 1.6]{MR2200122}) in the special case when \(Z\) has Nagata dimension zero. A general explicit version of Lang and Schlichenmaier's extension theorem has been obtained by Naor and Silberman \cite[Corollary 5.2]{MR2834732}.

As a special case of Corollary~\ref{cor:main2}, we get the following Lipschitz extension result for metric spaces where the metric assumes only a finite number of values. 
\begin{corollary}\label{cor:finite-dist-set}
Let \(X\), \(Y\) be complete metric spaces and \(Z\subset X\) a closed subset. Suppose that \(Y\) is \(\lambda\)-quasiconvex and the metric on \(Z\) assumes only finitely many values, that is,
\[
n\coloneqq \#\{d(z,z') \,:\, z, z'\in Z\}<\infty.
\] 
Then any \(1\)-Lipschitz map \(f\colon Z\to Y\) can be extended to a \(2^{n+1}\lambda  \)-Lipschitz map \(\bar{f}\colon X\to Y\).
\end{corollary}
Here, we use \(\# S\) to denote the cardinality of a set \(S\).

\subsection{Sharpness of Theorem~\ref{thm:main1}}Next, we discuss how sharp the constant \(8c\) appearing in Theorem \ref{thm:main1} is, both asymptotically and for specific values of \(c\).
We define \(c_N(X)\) to be the infimum of those \(c\geq 0\) for which \(X\) has Nagata dimension zero with constant \(c\). Theorem~\ref{thm:main1} is asymptotically sharp with regard to \(c_N(X)\) (up to a multiplicative constant). In fact, if \(C_n\) is the \(n\)-cycle graph equipped with the shortest-path metric, then we find \(c_N(C_n)=\frac{n}{2}\). Moreover, \(d_{\bl}(C_n, Y)\geq n-1\) for every weighted tree \(Y\) (see \cite{MR1486638}). Here, we use the notation
\[
d_{\bl}(X,Y)\coloneqq\inf\big\{\Lip(\psi) \Lip(\psi^{-1}) : \, \psi\colon X\to Y \text{ bijective}\big\}
\]
to denote the \textit{bilipschitz distance} between two metric spaces \(X\) and \(Y\). By the above, it follows that Theorem~\ref{thm:main1} is asymptotically sharp. 

We proceed to consider the case when \(c_N(X)=1\). It is a standard fact that \(X\) is an ultrametric space if and only if \(c_N(X)\) is attained and \(c_N(X)=1\). Let \(B_n\) denote the full binary tree of height \(n\) equipped with the shortest-path distance and \(X_n\subset B_n\) its leaves; see Example~\ref{ex:binary-tree} and Section~\ref{sec:lower-bound} for definitions. It well-known that \(X_n\) is an ultrametric space for every \(n\geq 1\). Due to the following result, the constant 8 appearing in Theorem~\ref{thm:main1} is the best possible for
finite ultrametric spaces.

\begin{theorem}[Chan, Xia, Konjevod and Richa \cite{MR2304999}]\label{thm:lowerBound}
Let \(X_n\) denote the leaves of the full binary tree \(B_n\). Then for every \(K\in [1, 8)\), there exists an integer \(n(K)\geq 1\) such that \[
d_{\bl}(X_{n(K)}, Y) \geq K
\]
for every weighted tree \(Y\).
\end{theorem}

In Section~\ref{sec:lower-bound}, we give a new proof of Chan, Xia, Konjevod and Richa's theorem using different methods. In contrast to their approach we do not consider minor mappings \(f\colon B_n\to X_n\) and only work with notions intrinsic to \(X_n\). Essentially, to prove Theorem~\ref{thm:lowerBound} we use a bootstrapping argument involving two quantities \(\rad(s)\) and \(\Conn(s)\); see Definitions~\ref{def:rad} and \ref{def:conn}. This strategy has the advantage that it can potentially be adapted to families of metric spaces with \(c_N(X)>1\); see Remark~\ref{rem:generalization-to-c-bigger-than-zero}.  
Moreover, our proof of Theorem~\ref{thm:lowerBound} yields the following exponential upper bound on \(n(K)\):
\begin{equation}\label{eq:exponential}
n(K) \leq \exp\Big(\frac{8000}{(8-K)^2}\Big).
\end{equation}

Most likely \eqref{eq:exponential} does not have the right order of growth. As an example, one can take \(n(1)=1\), \(n(2)=2\) and \(n(3)=3\). However, \(n(4)\geq 6\) and obtaining sharp bounds on \(n(K)\) seems to be a difficult problem. 

How sharp is Theorem~\ref{thm:main1} for other values of the Nagata constant \(c_N(X)\)? On account of the following result due to Gupta, if \(X\) is a finite \(0\)-hyperbolic space with \(c_N(X)>1\), then Theorem~\ref{thm:main1} cannot be sharp. 

\begin{theorem}[Gupta \cite{MR1958411}]\label{thm:gupta}
If \(X\) is a \(0\)-hyperbolic finite metric space, then there exists a weighted tree \(Y\) such that \(d_{\bl}(X,Y)\leq 8\).
\end{theorem}

See Section~\ref{sec:R-trees} for the definition of \(0\)-hyperbolic spaces. The leaves of a tripod whose edges have different lengths are a simple example of a \(0\)-hyperbolic finite metric space \(X\) satisfying \(c_N(X)>1\). Hence, for such a space we find that Theorem~\ref{thm:main1} is not sharp. But in general we do not know if Theorem~\ref{thm:main1} is optimal for arbitrary metric spaces \(X\)  with \(c_N(X)>1\).

\subsection{A generalization of Gupta's theorem} Every \(0\)-hyperbolic metric space admits an isometric embedding into a complete \(\R\)-tree (see \cite[Theorem 8]{MR753872}).
Thus, Gupta's theorem (Theorem~\ref{thm:gupta}) can be rephrased by saying that every finite subspace of an \(\R\)-tree is \(8\)-bilipschitz equivalent
to a weighted tree.  This has interesting consequences in theoretical computer science (see, for example, \cite{MR2291000}, \cite{MR2257250}). A metric space is called \textit{proper} if every closed ball of finite radius is compact. Our next result is a generalization of Gupta's theorem and an improvement of Corollary~\ref{cor:biLip} for  \(0\)-hyperbolic proper metric spaces:

\begin{theorem}\label{thm:guptaGen}
Every \(0\)-hyperbolic proper metric space \(X\) admits an \(8\)-bilipschitz embedding \(f\colon X\to Y\) into a complete \(\R\)-tree \(Y\) such that \(\sing Y \subset f(X)\).
\end{theorem}

In \cite{MR1056175},  Matou\v{s}ek proved that there exists a universal constant \(C>0\) such that every \(1\)-Lipschitz map \(f\colon Z\to E\) from a subset \(Z\subset X\) of an \(\R\)-tree \(X\) to a Banach space \(E\) can be extended to a \(C\)-Lipschitz map \(\bar{f}\colon X\to E\). As a direct corollary of Theorem~\ref{thm:guptaGen}, we can strengthen Matou\v{s}ek's result  in the special case when \(Z\) is
proper as follows:

\begin{corollary}\label{cor:matousek}
Let  \(X\), \(Y\) be metric space and \(Z\subset X\) a \(0\)-hyperbolic proper metric space. Suppose that \(Y\) is \(\lambda\)-quasiconvex.
Then any \(1\)-Lipschitz map \(f\colon Z\to Y\) admits an \(8\lambda\)-Lipschitz extension \(\bar{f}\colon X\to Y\).
\end{corollary}

We do not know if the constant \(8\lambda\) appearing in Corollary~\ref{cor:matousek} is sharp in general. In the following, we give a short outline of the proof of Theorem~\ref{thm:guptaGen}. The proof is done in three steps. First, we show that if \(\sep X> 0\), where 
\[
\sep X\coloneqq \inf \{ d(x,x') : x, x'\in X,\, x\neq x'\},
\]
and there are no pairwise distinct $x$, $y$, $z\in X$ satisfying $d(x,y)+d(y,z) = d(x,z)$, then Theorem~\ref{thm:guptaGen} is valid. To show this, we use a Whitney-type decomposition of \(E(X)\setminus X\), where \(E(X)\) denotes the injective hull of \(X\) (see Section~\ref{sec:injective-hull} for the definition). This decomposition has been used by Gupta in \cite{MR1958411} and it also appears in the work of Matou\v{s}ek \cite{MR1056175}. 

Second, we suppose that
\(\sep X>0\) and consider the decomposition of \(E(X)\setminus X=\bigcup_{\alpha \in A} U_\alpha\) into connected components. Since \(E(X)\) is an \(\R\)-tree (see Proposition~\ref{prop:dress-1}), \(X_\alpha\coloneqq \partial U_\alpha\) satisfies the assumptions of the first step, and thus we obtain for all \(\alpha\in A\) an \(8\)-bilipschitz embedding $f_\alpha: X_\alpha \to Y_\alpha$ into an \(\R\)-tree \(Y_\alpha\) such that \(\sing Y_\alpha \subset f_\alpha(X_\alpha)\). Now, \(Y\) can be obtained via a gluing construction (see Lemma~\ref{lem:gluing-sep-bigger-than-zero}). Finally, as a last step, we employ an ultrafilter argument (see Lemma~\ref{lem:ultrafilter}) to get rid of the assumption \(\sep X>0\). 

\subsection{Generalizations to higher dimensions?} To any metric space \(Z\) one can associate the weighted complete graph \(\mathcal{C}_Z=(Z, E, \omega)\), where \(E=\bigl\{ e\subset Z : \abs{e}=2\bigr\}\) and \(\omega(\{z,z'\})=d(z,z')\) for all \(\{z, z'\}\in E\). Let \(K_n\) denote the complete graph on \(n\) vertices.
We close the introduction with the following natural question:

\begin{question}\label{qe:question-1}
Fix \(n\geq 0\). Does there exist \(D_n>0\) such that whenever a metric space \(X\) has Nagata dimension less than or equal to \(n\) with constant \(c\geq 1\), there exist a dense subset \(Z\subset X\) and a \(K_{n+3}\)-minor free spanning graph \(H\subset \mathcal{C}_Z\) such that \((Z, d)\) and \((Z, d_H)\) are \(D_n c\)-bilipschitz equivalent?
\end{question}
Every graph-theoretic tree excludes \(K_3\) as a minor. Hence, it follows that \(D_0=8\) by Theorem~\ref{thm:main1}. A positive answer to Question~\ref{qe:question-1} leads directly to a quantitative version of the Lang--Schlichenmaier Lipschitz extension theorem \cite{MR2200122}. In fact, in this case one can use a result of Lee and Naor (see \cite[Theorem 1.7]{MR2129708}) to prove that \(\textrm{ae}(X)\ \lesssim D_n (c+1) (n+1)^2\) whenever \(X\) is complete and has Nagata dimension \(\leq n\) with constant~\(c\). Here, we use \(\textrm{ae}(X)\) to denote the absolute extendability constant of \(X\); see \cite{MR2129708} for more information.

\subsection{Acknowledgments}
We are indebted to Urs Lang and Stefan Wenger for reading and suggesting improvements to this text. The first named author would like to thank Anna Bot for helpful discussions regarding the exposition of this article. We are thankful to the referee for bringing \cite{MR2304999} to our attention. We also thank the referee for a very thorough report that allowed us to correct many inaccuracies in the text.
This research was supported by Swiss National Science Foundation Grant no. 182423.

\section{Preliminaries}\label{sec:preliminaries}

\subsection{Metric notions}
Let \((X,d)\) be a metric space. We use the following standard notation: 
\begin{align}
U(x,s)&\coloneqq\{ x'\in X : d(x,x')< s \}, &B(x,s)&\coloneqq\{ x'\in X : d(x,x')\leq s \}
\end{align}
to denote the open and closed ball of radius \(s\) centered at \(x\), respectively. Furthermore,
\begin{align}\label{eq:defBall}
S(x,s)&\coloneqq\{ x'\in X : d(x,x')=s \} 
\end{align}
denotes the sphere of radius \(s\) centered at \(x\). A metric space \(X\) is called \textit{proper} if every closed ball in \(X\) is compact. An \textit{ultrametric space} is a metric space \(X\) such that 
\[
d(x,y)\leq \max\{ d(x,z), d(z,y)\}
\]
for all \(x\), \(y\), \(z\in X\). 
It is a standard fact (see, for example, \cite[p. 8]{semmes2007introduction}) that \(X\) is an ultrametric space if and only if for every \(s\geq 0\) and every ball \(B=B(x,s)\subset X\), one has \(B=B(x',s)\) for all \(x'\in B\). 

A map \(f\colon X\to Y\) between metric spaces is \textit{\(L\)-Lipschitz}, for some constant \( L\geq 0\), if \(d(f(x), f(y))\leq L d(x,y)\) for all \(x\), \(y\in X\). The infimum of those \(L\geq 0\) for which \(f\) is \(L\)-Lipschitz is denoted by \(\Lip(f)\) and called \textit{Lipschitz constant} of \(f\). We write \(\Lip_{1}(X,Y)\) for the set of all \(1\)-Lipschitz maps from \(X\) to \(Y\). 
We say that \(f\colon X\to Y\)  is \textit{\(C\)-bilipschitz}, for some constant \(C\geq 0\), if there exists \(s>0\) such that 
\[
s\,d(x,y) \leq d(f(x), f(y)) \leq Cs \,d(x,y)
\]
for all \(x\), \(y\in X\). The smallest constant \(C\geq 0\) such that \(f\) is \(C\)-bilipschitz is denoted by \(\dist(f)\) and called \textit{distortion} of \(f\).

\subsection{Injective hulls}\label{sec:injective-hull}
We say that $Y$ is an \textit{injective metric space}, if for every metric space $X$ and every \(f\in \Lip_1(A, Y)\), where $A\subset X$, there is \(\hat{f}\in \Lip_1(X,Y)\) such that $\hat{f}|_A= f$. Injective metric spaces (also called hyperconvex metric spaces) have been introduced by Aronszajn and Panitchpakdi in \cite{MR84762}. A deep result of Isbell \cite{MR182949} shows that every metric space $X$ has an essentially unique \textit{injective hull} $(E(X), e)$, that is, \(E(X)\) is an injective metric space and \(e\colon X\to E(X)\) an isometric embedding, such that every isometric embedding $\iota\colon X\to Y$ into an injective metric space $Y$ factors through $e$. Consequently, if \(Y\subset E(X)\) is an injective metric space containing \(e(X)\), then \(Y=E(X)\). Isbell's construction has been rediscovered by Dress (see \cite{MR753872}) and is a well-known concept in phylogenetic analysis (see \cite{MR1379369} for a survey). Throughout this article, we will sometimes suppose that metric spaces are already isometrically embedded subsets of their injective hulls.

\subsection{Graph theory}
In the following, we recall standard terminology from graph theory as found in \cite{MR2159259} and \cite{MR607504}.
A \textit{graph} \(G\) is a pair \((V, E)\) consisting of a (possibly infinite) set \(V=:\vertices G\) and a subset \(E\subset \{ \{x,x'\} : x, x'\in V, \, x\neq x' \}\). A \textit{path} in \(G\) is a finite sequence \((x_0, x_1, \dots, x_n)\), where \(n\geq 0\), consisting of elements of \(V\) such that \(\{x_i, x_{i+1}\}\in E\) for all \(0\leq i \leq n-1\). We say that \(G\) is \textit{connected} if any two vertices \(x\),\(x'\in \vertices G\) can be connected via a path. A path  \((x_0, x_1, \dots, x_n)\)  in \(G\) is called \textit{cycle} if \(n\geq 3\), \(x_0=x_n\), and \(x_i\neq x_j\) for all \(1\leq i < j \leq n\). A graph having no cycles is called \textit{acyclic}. A \textit{tree} is a connected acyclic graph. A subgraph \(T\subset G\) is called \textit{spanning tree} if \(T\) is a tree and \(\vertices T=\vertices G\). As a consequence of the axiom of choice, every connected graph has a spanning tree (see, for example, \cite[Proposition 8.1.1]{MR2159259}). Let \(T\) be a tree. By definition, for any two points $x$, \(x'\in \vertices T\) there is a unique path \([x,x']_T\coloneqq(x_0, x_1, \dots, x_n)\) from \(x\) to \(x'\), such that \(x_i\neq x_j\) for all \(i\neq j\).

\subsection{Weighted trees}\label{sec:GraphTheory}
Let $(X,d)$ be a metric space and $T=(X,E)$ a tree. For every function $\omega:E\to (0,\infty)$ let \(d_\omega\) denote the metric on \(X\) given by
\begin{equation}\label{eq:weight-function-1}
d_{\omega}(x,x')\coloneqq\sum_{i=0}^{n-1} \omega(\{x_{i}, x_{i+1}\}),  
\end{equation}
where \([x,x']_T=(x_0,\dots, x_n)\) provided \(x\neq x^\prime\). The metric \(d_\omega\) is called \textit{shortest-path metric} of \((T, \omega)\). We set \(X_\omega\coloneqq (X, d_\omega)\). Moreover, if \(\omega(\{ x, x'\})=d(x,x')\) for all \(\{x, x'\}\in E\), then we define \(d_T\coloneqq d_\omega\) and \(T[X]\coloneqq (X, d_T)\). Clearly, \(X_\omega\) is a weighted tree and for every weighted tree \(X\), by definition, there is a tree \(T=(X, E)\) and a weight function \(\omega\colon E\to (0, \infty)\) such that \(X=X_\omega\).

 To find the best approximation of \(X\) by a weighted tree it suffices to consider weighted trees of the form \(T[X]\). This is a direct consequence of the following proposition.

\begin{proposition}
\label{prop:optimalWeights}
Let $X$ be a metric space and $T=(X,E)$ a tree. Then the identity map $X \to T[X]$ has minimal distortion among all identity maps  \(X\to X_\omega\), where \(\omega\) ranges over all functions \(\omega\colon E\to (0, \infty)\). In particular,
one has
\[
\inf\Bigl\{\,d_{\bl}(X,Y)\, :\, Y \text{ weighted tree}\,\Bigr\} =\inf\Bigl\{\, d_{\bl}(X, T[X])\, :\, T=(X, E)\, \text{tree}\,\Bigr\}
\]
for every metric space \(X\). 
\end{proposition}

\begin{proof}
Fix $\omega\colon E\to (0,\infty)$. We put
$$
s\coloneqq \inf\bigg\{\frac{d_\omega(x,x')}{d(x,x')} : \{x,x'\} \in E\bigg\}.
$$
If $s = 0$, then the identity map to $(X,d_\omega)$ is not Lipschitz, so there is nothing to prove. Otherwise, we get for all $x,x'\in X$, 
$$
s d(x,x') \leq s d_T(x,x') = s\sum_{i=0}^{n-1} d(x_i,x_{i+1}) \leq \sum_{i=0}^{n-1} \omega(\{x_i,x_{i+1}\}) = d_\omega(x,x'),
$$
where $(x_0,\dots, x_n) = [x,x']_T$. Therefore,
$$
\sup_{x,x'\in X}\frac{d_\omega(x,x')}{d(x,x')}\geq s \sup_{x,x'\in X}\frac{d_T(x,x')}{d(x,x')},
$$
from which the claim follows.
\end{proof}

\subsection{\texorpdfstring{$\R$}{TEXT}-trees}\label{sec:R-trees}
A metric space $X$ is \textit{geodesic} if for every pair of points \(x\), \(x'\in X\) there exists an isometric
embedding \(\sigma\colon [a, b]\to X\) of an interval \([a, b]\subset \R\) into $X$ such that \(\sigma(a)=x\) and \(\sigma(b)=x'\). We denote the image of such an embedding by $[x,x']_X$. 
A metric space \(X\) is $0$-\textit{hyperbolic} if
$$
d(w,x) + d(y,z) \leq \max\{d(w,y) + d(x,z) , d(w,z) + d(x,y)\}
$$
for all \(x\), \(y\), \(z\), \(w\in X\). Metric subspaces of $0$-hyperbolic spaces are $0$-hyperbolic, and any \(\GoodName\) is \(0\)-hyperbolic. 
An \textit{$\R$-tree} is a geodesic and $0$-hyperbolic metric space. Further information concerning $0$-hyperbolic metric spaces can be found in \cite{MR753872}.

Notice that a geodesic metric space is $0$-hyperbolic if and only if all of its geodesic triangles are tripods. In particular, any \(\R\)-tree \(X\) is uniquely geodesic and thus there is no ambiguity in the notation $[x,x']_X$.
Every complete \(\R\)-tree is an injective metric space (see, for example, \cite[Theorem 8]{MR753872}).
By a result of Dress, there is the following connection between \(0\)-hyperbolic spaces and \(\R\)-trees:
\begin{proposition}[Theorem 8 in \cite{MR753872}]\label{prop:dress-1}
Let \(X\) be a metric space. Then the following are equivalent:
\begin{enumerate}
    \item \(X\) is \(0\)-hyperbolic
    \item  $E(X)$ is an $\R$-tree.
\end{enumerate}
In particular, every \(0\)-hyperbolic space admits an isometric embedding into a complete \(\R\)-tree.
\end{proposition}

A subset \(A\subset X\) of an \(\R\)-tree \(X\) is \textit{convex} if \([a,a']_X\subset A\) for all \(a\), \(a'\in A\). 
The following description of the convex hull of a subset of an \(\R\)-tree is needed in the proof of Corollary~\ref{cor:biLip}.

\begin{lemma}\label{lem:convex-hull}
For every subset \(A\subset X\) of an \(\R\)-tree \(X\), the set
\[
G_1(A)\coloneqq\bigcup_{a,a'\in A} [a,a']_X
\]
is convex. 
Moreover, if \(X\) is complete, then the closure of \(G_1(A)\) is isometric to \(E(A)\).
\end{lemma}
\begin{proof}
By the definition of a convex set, it suffices to consider the case when \(A\) consists of four points. If \(A=\{a_1, \dots, a_4\}\), then (by relabeling the points if necessary) we find that there are \(p_1\), \(p_2\in X\), such that 
\begin{equation}\label{eq:darstellung-1}
G_1(A)=\bigcup_{i=1,2} [a_i,p_1]_X\cup[p_1,p_2]_X\cup \bigcup_{i=3,4} [p_2,a_i]_X.
\end{equation}
Since the union of two convex subsets of \(X\) that have a point in common is convex, it follows directly from the decomposition \eqref{eq:darstellung-1} that \(G_1(A)\) is convex. 

Suppose now that \(X\) is a complete \(\R\)-tree.
Since \(X\) is injective, without loss of generality we may suppose that \(X=E(A)\). By construction, the closure of \(G_1(A)\) in \(X\), which we denote by \(G_1(A)^-\), is a strongly convex closed subset of \(X\). See \cite[p. 102]{MR3349339} for the definition of a strongly convex subset of a metric space. Using Lemmas 3.2 and 3.3 of \cite{MR3349339}, we obtain that \(G_1(A)^-\) is an injective metric space. As, \(A\subset G_1(A)^- \subset E(A)\), by definition of the injective hull, we infer that the closure of \(G_1(A)\) in \(X\) is isometric to \(E(A)\), as desired. 
\end{proof}

\subsection{Partial orders}\label{sec:PartialOrder}
Given a partial order \(\preceq\) on a set \(M\), we denote by 
\[
\up_{\preceq}[x]\coloneqq \{x'\in X : x \preceq x' \}
\]
the \textit{upper closure} of \(x\in M\). Further, for \(A\subset M\) we denote by 
\[
\min_{\preceq} A\coloneqq\big\{a\in A : \text{there exists no } a'\in A \setminus \{a\}, \text{ such that } a' \preceq a \big\}
\]the set of minimal elements of \(A\). The set \(\max_{\preceq} A\) is defined analogously. A point \(x\in M\) is called \textit{lower bound} of \(A\) if \(x\preceq a\) for all \(a\in A\). Notice that if \(a\in \min_{\preceq}A\), then \(a\) is not necessarily a lower bound of \(A\).   

Now, let \(T\) be a tree and \(o\in \vertices T\). 
The relation \(x \preceq_{o, T} x'\) if and only if \(x\in [o, x']_T\) defines a partial order on \(\vertices T\). Notice that \(\min_{\preceq_{o, T}} A\neq \varnothing\) provided \(A\neq \varnothing\). Further, if $Y$ is an $\R$-tree and $o\in Y$, then, as for trees,  $x \preceq_{o,Y} x'$ if and only if $x \in [o,x']_Y$ defines a partial order on $Y$. 

\subsection{Ultraproducts}
Let \(\{ (X_i, p_i) : i\in I \}\) be a family of pointed metric spaces and set
\[
\ell_\infty(I, X_i)\coloneqq\Bigl\{ (x_i)_{i\in I} \in \prod_{i\in I} X_i \, : \, \sup_{i\in I} d(x_i, p_i) < +\infty \Bigr\}.
\]
Suppose that \(\mathfrak{U}\) is a free ultrafilter on \(I\).  The \textit{ultraproduct} \((X_i)_{\mathfrak{U}}\) of the family \(\{ (X_i, p_i) : i\in I\}\) is the set of all equivalence classes of \(\ell_\infty(I, X_i)\) under the equivalence relation
\[
(x_i)_{i\in I} \sim (x_i^\prime)_{i\in I}\quad \quad \text{   if and only if   } \quad \quad \lim_{\mathfrak{U}} d(x_i, x^\prime_i)=0, 
\]
endowed with the metric 
\[
d((x_i)_\mathfrak{U}, (x_i^\prime)_\mathfrak{U})=\lim_{\mathfrak{U}} d(x_i, x_i^\prime),
\]
where \((x_i)_\mathfrak{U}\) and \((x_i^\prime)_{\mathfrak{U}}\) denote the equivalence classes of \((x_i)_{i\in I}\) and \((x_i^\prime)_{i\in I}\), respectively. Ultraproducts have striking applications in Banach space theory and metric geometry (see, for example, \cite{MR552464}, \cite{MR1608566}). 
The following lemma is employed in the proof of Theorem \ref{thm:guptaGen}.
\begin{lemma}
\label{lem:ultrafilter}
Let \(\{ (Y_i, p_i) : i\in \N \}\) be a sequence of pointed \(\R\)-trees and suppose \(\mathfrak{U}\) is a free ultrafilter on \(\N\). Then \(Y\coloneqq (Y_i)_{\mathfrak{U}}\) is a complete \(\R\)-tree and \(\sing Y\subset ( \sing Y_i )_{\mathfrak{U}}\).
\end{lemma} 

\begin{proof}
A standard argument shows that the ultraproduct of pointed \(\CAT(\kappa)\)-spaces is a complete \(\CAT(\kappa)\)-spaces (see, for example, \cite[Lemma 2.4.4]{MR1608566}). A metric space is an \(\R\)-tree if and only if it is a \(\CAT(\kappa)\)-space for every \(\kappa\in \R\) (see \cite[p. 167]{MR1744486}). Hence, \(Y\) is a complete \(\R\)-tree. In what follows, we show that \(\sing Y\subset ( \sing Y_i )_{\mathfrak{U}}\). 
Fix \(x=(x_i)_{\mathfrak{U}}\in Y\). For every \(x_i\) we denote by \(r(x_i)\geq 0\) the supremum of those \(r\geq 0\) for which there is an isometric embedding \(\phi\colon U(x_i,r)\to (-r,r)\). 

Suppose that \(2r\coloneqq \lim_{\mathfrak{U}} r(x_i)>0\).
We set \(A\coloneqq \{ i\in \N : r(x_i) > r \}\) and for each \(i\in A\) let \(\phi_i\colon U(x_i,r)\to (-r,r)\) be an isometry with \(\phi_i(x_i)=0\). Notice that, by construction, \(A\in \mathfrak{U}\). It is not hard to check that \(U(x,r)\) is contained in \((U(x_i, r))_{\mathfrak{U}}\) and the restriction of \(\phi=(\phi_i)_{\mathfrak{U}}\) to \(U(x,r)\) is an isometry onto \((-r,r)\) mapping \(x\) to the origin. Hence, \(x\) is a regular point of \(Y\). 
Now, suppose that \(\lim_{\mathfrak{U}} r(x_i)=0\). For every \(x_i\) with \(r(x_i)<+\infty\) choose a singular point \(s_i\in Y_i\) such that \(d(x_i, s_i)\leq r(x_i)\). By construction, \((x_i)_{\mathfrak{U}}=(s_i)_{\mathfrak{U}}\) and thus \(x\in ( \sing Y_i )_{\mathfrak{U}}\). By putting everything together, we have shown that if \(x\) is a singular point of \(Y\), then  \(x\in ( \sing Y_i )_{\mathfrak{U}}\), as desired. 
\end{proof}

\section{Approximating spaces of Nagata dimension zero by weighted trees}\label{sec:main-result-and-examples}
In this section we prove Theorem~\ref{thm:main1} and its corollaries. We begin with an example showing that in general Theorem~\ref{thm:main1} is not valid if one requires that the weighted tree is \(8c\)-bilipschitz equivalent to \(X\).
\begin{example}\label{ex:2-adic-numbers}
Let \(X=\Z_p\) denote the \(p\)-adic integers for some prime number \(p\). In what follows, we show that there is no weighted tree \(Y\) such that \(X\) and \(Y\) are bilipschitz equivalent. By Proposition~\ref{prop:optimalWeights}, it suffices
to show that there is no tree \(T=(X,E)\) such that \(T[X]\) and \(X\) are bilipschitz equivalent.
To this end, let \(T=(X,E)\) be any tree. There is a point \(x\in X\), such that the neighborhood 
\[
N_T(x)\coloneqq\big\{ x'\in X \, : \, \text{\(x\) and \(x'\) are adjacent in \(T\)}\big\}
\]
of \(x\) in \(T\) contains uncountably many points. Indeed if this is not the case, then \(X\) is the countable union of countable sets. But this is not possible, since X is uncountable.

Now, since \(X\) is a \(p^n\)-valued metric space, there exist an integer \(n\geq 0\) and an uncountable subset \(A\subset N_T(x)\), such that \(d(x,a)=p^{-n}\) for all \(a\in A\). By construction, \(d_T(a,a')=2\cdot p^{-n}\) for all distinct points \(a\), \(a'\in N_T(x)\). As \(X\) is separable and \(A\) is uncountable, the set \(A\) contains points that are arbitrarily close to each other, and thus 
\[
\sup_{a\neq a'} \frac{d_T(a,a')}{d(a,a')}=\infty.
\]
It follows that \(X\) and \(T[X]\) cannot be \(C\)-bilipschitz equivalent for any \(C\geq 1\).
Hence, \(X\) is not bilipschitz equivalent to any weighted tree.
\end{example}

Before we proceed with the proof of Theorem~\ref{thm:main1}, we introduce some notation. 
A pair \((G, r_0)\) consisting of a graph \(G\) and a point \(r_0\in \vertices G\) is called \textit{rooted graph}. 
We use the following `star' operation on rooted graphs. 
Let \(I\) be a set with \(0\in I\) and \(\mathcal{F}\coloneqq\{ (G_i, r_i) : i\in I \}\) a family of rooted graphs. 
The rooted graph \(\Star(G_0; \mathcal{F})\) is by definition the graph with root \((0, r_0)\),  vertex set
\[V=\bigcup_{i\in I} \,\bigl\{ (i, v) : v\in V_i \bigr\},\]
and the edge set \(E\) is given by the rule \(\{ (i, v), (j, w) \} \in E \) if and only if 
\begin{itemize}
\item \(\text{if } i=j,\text{ then } \{v, w\} \in E_i\), 
\item \(\text{if } i\neq j, \text{ then } (i,v)=(0,r_0) \text{ and } (j,w)=(j,r_j)\) or vice versa.
\end{itemize}
For example, if each of the graphs \(G_i\in \mathcal{F}\) is a singleton graph, then the rooted graph \(\Star(G_0; \mathcal{F})\) is a star graph with central vertex \(r_0\). Moreover, if each of the graphs \(G_i\) is a tree, then \(\Star(G_0; \mathcal{F})\) is a tree as well.

We need the following characterization of spaces of Nagata dimension zero. Let \(X\) be a metric space and \(s>0\) a real number. Following David and Semmes (see \cite[p. 158]{MR1616732}), we call a sequence of points \(x_1, \dots, x_k\in X\) an \textit{\(s\)-chain} if \(d(x_{i}, x_{i+1})\leq s\) for all \(1 \leq i <k\). We write \(x\sim_s x'\) if \(x\) and \(x'\) can be connected via an \(s\)-chain. Clearly, \(\sim_s\) is an equivalence relation on \(X\).
We write \([x]_s\) for the equivalence class of \(x\in X\) under the relation \(\sim_s\). Notice that \(X\) has Nagata dimension zero with constant \(c\) if and only if for every \(s>0\), one has \(\diam [x]_s \leq cs\) for all \(x\in X\). Let \(\mathcal{B}\) be a collection of subsets of \(X\). A function \(p\colon \mathcal{B}\to X\) is called \textit{choice function} if \(p[B]\in B\) for all \(B\in \mathcal{B}\). Now we are in position to prove our main theorem.
\begin{proof}[Proof of Theorem \ref{thm:main1}]
For each \(i\in \Z\), we put \(s_i\coloneqq 2^i\) and write
\[
\mathcal{B}_{i}\coloneqq \bigl\{ [x]_{s_i} : x\in X \bigr\}
\]
for the partition of \(X\) into equivalence classes of \(\sim_{s_i}\).
Notice that \(\diam B\leq c 2^i\) for every \(B\in \mathcal{B}_i\). Given \(B\in \mathcal{B}_i\), let \(\pred^{(i)}[B]\subset \mathcal{B}_{i-1}\) denote the subset consisting of all \(B'\in \mathcal{B}_{i-1}\) with \(B'\subset B\). Fix \(o\in X\) and set \(B_{o,i}\coloneqq [o]_{s_i}\in \mathcal{B}_i\) for all \(i\in \Z\).

For every \(i\in \Z\) there is a choice function \(p^{(i)}\colon \mathcal{B}_i\to X\) satisfying \(p^{(i)}[B_{o,i}]=o\) and 
\[
p^{(i)}[B]\in \Bigl\{ p^{(i-1)}[B'] : B^\prime \in \pred^{(i)}[B] \Bigr\}
\] 
for every $B\in \mathcal{B}_i$. In order to see this, first fix a choice function $p^{(0)}\colon \mathcal{B}_0\to X$ with $p^{(0)}(B_{o,0}) = o$ using the axiom of choice. The claim then follows by induction, once for the positive integers and the negative integers, respectively. Notice that, in general, \(\mathcal{B}_i\cap \mathcal{B}_{i+1}\neq \varnothing\) and \(\pred^{(i)}[B]\neq \pred^{(i+1)}[B]\) 
for \(B\in \mathcal{B}_i\cap \mathcal{B}_{i+1}\).  
To ease notation, we set
\[
\mathcal{B}\coloneqq \bigcup_{i\in \Z} \mathcal{B}_i\times \{i\}
\]
and define the function $p:\mathcal{B}\to X$ by setting $p[(B,i)]= p^{(i)}[B]$. Further, we let \(i[B]\in \Z\) denote the second coordinate of \(B\in \mathcal{B}\), and we put \(\pred[(B,i)]\coloneqq\pred^{(i)}[B]\times \{i-1\}\) for all \((B, i)\in \mathcal{B}\). By construction, \(\pred[B]\subset \mathcal{B}\) for all \(B\in \mathcal{B}\).

Given \(B\in \mathcal{B}\) we denote by \(T^{0}_B\) the rooted singleton graph \((\{p[B]\}, \varnothing, p[B])\).
We define the rooted tree \(T_{B}^{\ell+1}\) by induction on $\ell\geq 0$ as follows. Fix \(B\in \mathcal{B}\) and consider the family \(\mathcal{F}\coloneqq\bigl\{ T_{B'}^{\ell} : B'\in \pred[B]\bigr\}\).
We set 
\[T^{\ell+1}_B\coloneqq \Star\big(T_{B_{0}^\prime}^{\ell}; \mathcal{F}\big),
\] where \(B_{0}^\prime\in \pred[B]\) is the unique member such that \(p[B]\) is contained in  \(B_{0}^\prime\) (or, more precisely, the first coordinate of \(B_{0}^\prime\in \mathcal{B}\) contains \(p[B]\)). Notice that, by construction, the rooted tree \(T_B^{\ell}\) has height at most \(\ell\). Let \(d_{T^{\ell}_{B}}\) denote the shortest-path metric of the weighted graph \((T^{\ell}_{B}, \omega)\), where \(\omega(\{x,x'\})\coloneqq d(x,x')\) for all \(\{x,x'\}\in \edges T^{\ell}_{B}\).
For the remainder of this proof, we will use the notation $X_B^\ell := \vertices T_B^\ell\subset X$. We claim that for all \(\ell\geq 0\) and all \(B\in \mathcal{B}\),
\begin{equation}\label{eq:ineq2}
d_{T^{\ell}_{B}}(x,p[B]) \leq c \,2^{i[B]} \, \sum_{k=0}^{\ell-1} 2^{-k}
\end{equation}
for all \(x\in X^{\ell}_B\).  Clearly, if \(\ell=0\), then \eqref{eq:ineq2} is valid for all \(B\in \mathcal{B}\). Now, fix \(\ell>0\) and \(B\in \mathcal{B}\). For every \(x\in X_B^{\ell}\) there is a unique \(B'\in \pred[B]\) such that \(x\in B'\), so by setting \(i\coloneqq i[B]\), we obtain by construction of \(T_B^{\ell}\), 
\begin{equation*}\label{eq:ineq3}
d_{T_B^{\ell}}(x,p[B])=d(p[B], p[B'])+d_{T_{B'}^{\ell-1}}(x, p[B']) \leq c 2^{i}+c 2^{i-1} \sum_{k=0}^{\ell-2} 2^{-k},
\end{equation*}
as claimed. Next, we show that for all \(\ell \geq 0\) and all \(B\in \mathcal{B}\),
\begin{equation}\label{eq:ineq1}
d(x_1, x_2)\leq d_{T_B^{\ell}}(x_1,x_2)\leq 8c \, d(x_1, x_2).
\end{equation}
for all \(x_1\), \(x_2\in X_B^{\ell}\). This is clearly the case when \(\ell=0\). 
Fix \(\ell >0\) and \(B\in \mathcal{B}\), and let \(x_1\), \(x_2\in X_B^{\ell}\). By induction, we may suppose that \eqref{eq:ineq1} is valid if \(x_1\), \(x_2\in B'\) for some \(B'\in \pred[B]\). Now, suppose \(x_1\in B_1\) and \(x_2\in B_2\), where \(B_1\), \(B_2\in \pred[B]\) with \(B_1\neq B_2\).
Notice that this implies \(2^{i-1}\leq d(x_1,x_2) \), where \(i=i[B]\). By the use of \eqref{eq:ineq2}, we obtain
\begin{align*}
 d_{T_B^{\ell}}(x_1,x_2) &\leq d_{T_B^{\ell}}(x_1,p[B])+d_{T_B^{\ell}}(p[B],x_2)\leq 2c \, 2^i \sum_{k=0}^{\ell-1} 2^{-k}. 
\end{align*}
Hence, 
\[
d_{T_B^{\ell}}(x_1,x_2)\leq 2c \, 2^i\, \bigl( 2- 2^{-\ell+1}\bigr)\leq 8c\, \bigl( 1- 2^{-\ell}\bigr) \, d(x_1, x_2),
\]
as claimed. To finish the proof we construct a dense subset \(Z\subset X\) and a tree \(T=(Z,E)\) such that 
\[
d(x_1, x_2)\leq d_T(x_1, x_2)\leq  8c \, d(x_1, x_2)
\]
for all \(x_1\), \(x_2\in Z\). We define 
\[
T\coloneqq\bigcup_{i\geq 0} T^{2i}_{B_{o,i}} 
\]
and \(Z\coloneqq\vertices T=\bigcup_{i\geq 0} X^{2i}_{B_{o,i}}\). By definition of the choice function \(p\colon \mathcal{B}\to X\), it follows that \(T^{2j}_{B_{o,j}}\) is an induced subgraph of \( T^{2i}_{B_{o,i}}\) for all \(i\geq j\), and so \(T\) is well-defined. By construction, \(T\) is a tree and \(Z\subset X\) a dense subset. Fix \(x_1\), \(x_2\in Z\). There exists \(i\geq 0\) such that \(x_1\), \(x_2\in X_{B_{o,i}}^{2i}\). Hence, as \(T_{B_{o,i}}^{2i}\subset T\) is an induced subgraph, 
\[d_T(x_1, x_2)=d_{T_{B_{o,i}}^{2i}}(x_1, x_2).\] Therefore, by using \eqref{eq:ineq1}, we obtain \(d(x_1, x_2)\leq d_T(x_1, x_2)\leq  8c \, d(x_1, x_2)\), as desired. Thus \(Y\coloneqq (Z, d_T)\) and \(Z\) are \(8c\)-bilipschitz equivalent. This completes the proof.
\end{proof}

In what follows, we show that for $2^n$-valued ultrametric spaces the constant $8$ appearing in Theorem~\ref{thm:main1} can be improved to \(4\). A metric space \(X\) is \textit{\(2^n\)-valued}
if for every pair of distinct points \(x\), \(x'\in X\) there exists an integer \(n=n_{xx'}\in\Z\) such that \(d(x,x')=2^n\). The metric space of \(2\)-adic integers \(\Z_2\) is a prime example of a \(2^n\)-valued ultrametric space. 

\begin{proposition}\label{prop:2^n-valued}
For every  \(2^n\)-valued ultrametric space \(X\), there exists a dense subset \(Z\subset X\) and a weighted tree \(Y\) such that \(d_{\bl}(Z,Y)\leq 4\).
\end{proposition}
\begin{proof}
Let \(Z\subset X\) and the tree \(T=(Z,E)\) be constructed as in the proof of Theorem~\ref{thm:main1}.
Fix two distinct points \(x\), \(x'\in Z\) and let \(n\in \Z\) denote the unique integer such that \(d(x,x')=2^n\). By construction of \(T\), there exists \(j \in \N\) such that \(x\), \(x'\in X^{2j}_{B_{o,j}}\) and there exist \(x_0, \dots, x_m\), \(x_0', \dots, x_m'\in X\) such that, \(x_0=x\), \(x_0'=x'\), \(x_m=x_m'\), $d_T(x_{i-1}, x_i)\leq 2^{(n-m)+i}$, and $d_T(x'_{i-1}, x'_i)\leq 2^{(n-m)+i}$ for all \(i=1, \dots, m\). Consequently,
\[
d_T(x,x')\leq 2\sum_{i=n-m+1}^n 2^i<4\cdot 2^n=4\,d(x,x')
\]
and so \(d_{\bl}(Z,Y)\leq 4\), where \(Y\coloneqq (Z, d_T)\). This completes the proof.
\end{proof}
The following example of Gupta \cite[Theorem~6.1]{MR1958411} shows that Proposition~\ref{prop:2^n-valued} cannot be improved in general. We slightly adapted the construction.
\begin{example}\label{ex:binary-tree}
Let us briefly recall the construction of the full binary tree \(B_n=(V_n, E_n)\) of height $n\geq 0$. For $n= 0$, $B_0=\{r\}=:X_0$ is the singleton graph. Given $B_n$ and $X_n$, the tree $B_{n+1}$ is obtained from \(B_n\) by connecting every vertex in $X_n\subset B_n$ with two new edges to two new vertices. Let us denote by $X_{n+1}\subset V_{n+1}$ the newly created vertices. Note that $X_n$ is exactly the set of leaves of $B_n$ for every $n\geq 1$. The vertex $r\in B_0 \subset B_n$ is the root of $B_n$.

Now, fix $n\geq 2$ and define $G_n:=(W_n,F_n)= \Star(B_n, B_{n-1})$, i.e. a rooted graph, where the root has three children and each of the children is the root of a copy of $B_{n-1}$. The level of the root of $G_n$ is defined as $0$ and for every vertex $x\in W_n$, the level denotes the number of edges it passes on the unique shortest path to the root. The level of an edge $\{x,y\} \in F_n$ is defined as the maximum of the levels of $x$ and $y$. Let $Y_n$ denote the set of leaves of $G_n$.

We define the weight function \(\omega\colon F_n\to (0, \infty)\) as follows. If the level of the edge \(e\) is equal to \(n\), then we set \(\omega(e)\coloneqq 1\). Further, if the level of \(e\) equals \(k\in \{1, \dots, n-1\}\), then we set \(\omega(e)\coloneqq 2^{(n-1)-k}\). By construction of \(\omega\), it follows that if we equip \(Y_n\) with \(d_{\omega}\), then \(Y_n\) is a \(2^n\)-valued ultrametric space with \(\sep Y_n=2\) and \(\diam Y_n= 2^n\).
Because of Proposition~\ref{prop:optimalWeights}, there is a tree \(T=(Y_n, E)\) such that
\[
d_{\bl}(Y_n, T[Y_n])=\inf\big\{ d_{\bl}(Y_n, Y) \, : \, Y \text{ weighted tree}\big\}.
\]
From the proof of Proposition~\ref{prop:2^n-valued}, it follows that \(d_{\bl}(Y_n, T[Y_n]) < 4\). Thus, given \(x\), \(x'\in Y_n\), the path \([x,x']_T\) is contained in \(B(x, 2^k)\), where \(2^k=d_\omega(x,x')\). Indeed, if there exists $x_i\in [x,x']_T \setminus B(x,2^k)$, then $d_T(x,x') \geq d_\omega(x,x_i) + d_\omega(x_i,x')\geq 2\cdot 2^{k+1}=4 d_\omega(x,x')$.

As a consequence, if $\{x,x'\}\in E$ and $d_\omega(x,x')> 2^k$, then there exists no other $\{y,y'\}\in E$ with $y\in B(x,2^k)$ and $y'\in B(x',2^k)$. In particular, there exist precisely two edges in $T$ between the leaves of the three binary trees $B_{n-1}\subset G_n$. Let $x_1$, $y_1\in Y_n$ be the unique pair of vertices in the two copies of $X_{n-1}$ that are connected to vertices in the third copy of $X_{n-1}$. In particular, $d_T(x_1,y_1) \geq 2^{n+1}.$ Recursively, given the points $x_k,y_k$, define $x_{k+1}$ to be the unique point in $S(x_k, 2^{n-k})$ that is connected in $T$ to a point in $B(x_k, 2^{n-k-1})$. Define $y_{k+1}$ analogously. It follows that $d_\omega(x_n,y_n) = 2^n$, while
\[
d_T(x_n,y_n)\geq 2^{n+1}+ 2\sum_{i=1}^{n-1} 2^i= 4\cdot 2^n - 4 = 4\Big(1-\frac{1}{2^n}\Big) d_\omega(x,x').
\]
Consequently, \(d_{\bl}(T[X_n],X_n)\geq 4 \big(1-2^{-n}\big)\). As \(n\geq 2\) was arbitrary, this implies that the constant \(4\) appearing in Proposition~\ref{prop:2^n-valued} is sharp. 
\end{example}

We conclude this section by proving Corollaries~\ref{cor:biLip}, \ref{cor:main2} and \ref{cor:finite-dist-set}.

\begin{proof}[Proof of Corollary~\ref{cor:biLip}]
By Theorem~\ref{thm:main1} there is a dense subset \(Z\subset X\) and a tree \(T=(Z,E)\), such that the identity map \(\id\colon (Z,d)\to (Z, d_T)\) is \(8c\)-bilipschitz. 
Let \((Y, e)\) denote the injective hull of \((Z, d_T)\).  Proposition~\ref{prop:dress-1} tells us that \(Y\) is a complete \(\R\)-tree. Since \(Y\) is complete and \(Z\subset X\) is a dense subset, the continuous extension of \(e\circ \id\) to \(X\) is an \(8c\)-bilipschitz map which, by abuse of notation, we denote by \(e\colon X\to Y\). 

We claim that \(\sing Y\subset e(X)\).
Since \(Y\) is uniquely geodesic, for every vertex \(z_i\in [z,z']_T\) on the path in the tree $T$, one has \(e(z_i)\in [e(z), e(z')]_Y\). By combining this fact with Lemma~\ref{lem:convex-hull}, it follows that
\[
Y_0\coloneqq \bigcup_{\{z,z'\}\in E} [e(z), e(z')]_Y
\]
is a convex subset of \(Y\) containing \(e(Z)\). In the following, given \(A\subset Y\) we write \(A^-\) to denote the closure of \(A\) in \(Y\). 
Notice that \(Y_0^-=Y_0\cup e(Z)^-\). Since \(Z\subset X\) is a dense subset and \(X\) is complete, we obtain \(e(Z)^-=e(X)\) and thereby \(Y_0^-=Y_0\cup e(X)\). Thus, as \(Y_0^-\subset Y\) is an injective metric space containing \(e(Z)=e(T[Z])\), we have \(Y=Y_0^-=Y_0\cup e(X)\). 
Next, we fix $y\in Y_0\setminus e(X)$ and show that $y\notin \sing Y$. If $y\in e(Z)^-$, then $y\in e(X)$, since $X$ is complete. So, if $y\notin e(Z)^-$, then there exists $r>0$ such that $U_{Y_0}(y,r)$ is isometric to an open interval and does not intersect $e(Z)$. Since $Y$ is the completion of $Y_0$, this means that $U_Y(y,r)$ is also isometric to an open interval and therefore, $y\notin \sing Y$. We conclude that \(\sing Y \subset e(X)\), as was to be shown. 
\end{proof}
 
\begin{proof}[Proof of Corollary~\ref{cor:main2}]
We may suppose that \(Z\) is non-empty. By virtue of Corollary~\ref{cor:biLip}, there is an \(8c\)-bilipschitz map \(e\colon Z\to W\) into a complete \(\R\)-tree \(W\), such that \(\sing W\subset e(Z)\). By scaling (if necessary), we may assume that \(\Lip(e)=8c\) and \(\Lip(e^{-1})=1\), where \(e^{-1}\colon e(Z)\to Z\) denotes the inverse function of \(e\). Since \(W\) is injective, there exists an \(8c\)-Lipschitz map \(\bar{e}\colon X\to W\) such that \(\bar{e}(z)=e(z)\) for all \(z\in Z\). Therefore, if every function \(g\in \Lip_1(e(Z), Y)\) admits a \(\lambda\)-Lipschitz extension \(\bar{g}\colon W\to Y\), then every \(f\in \Lip_1(Z, Y)\) admits an \(8c \lambda\)-extension \(\bar{f}\colon X\to Y\). Indeed, by setting \(g\coloneqq f\circ e^{-1}\), which is a \(1\)-Lipschitz map and thus admits a \(\lambda\)-Lipschitz extension \(\bar{g}\colon W\to Y\), we obtain that \(\bar{f}\coloneqq \bar{g}\circ\bar{e}\) is an \(8c\lambda\)-Lipschitz extension of \(f\), as desired. 

To finish the proof, we must show that every \(g\in \Lip_1(e(Z), Y)\) admits a \(\lambda\)-Lipschitz extension \(\bar{g}\colon W\to Y\).
Since \(\sing W \subset e(Z)\) and \(e(Z)\subset W\) is a closed subset, every connected component \(W_\alpha\) of \(W\setminus e(Z)\) is isometric to a non-empty open interval \(I_\alpha\subset \R\) via an isometry \(\iota_\alpha\colon I_\alpha \to W_\alpha\). Since every \(w\in W_\alpha\) is a regular point, it follows that \(I_\alpha\) cannot be equal to \(\R\). In particular, it is either a bounded open interval or an open ray. In case it is an open ray, we have \(\partial W_\alpha=\{w_\alpha\}\) and we set \(\bar{g}(w)=g(w_\alpha)\) for all \(w\in W_\alpha\). If \(I_\alpha\) is bounded, then \(\partial W_\alpha=\{w_\alpha^-, w_\alpha^+\}\) with \(w_\alpha^-\neq w_\alpha^+\). Using that \(Y\) is \(\lambda\)-quasiconvex, we can choose for every \(\alpha\) a constant-speed curve \(\sigma_\alpha\colon I_\alpha \to Y\) connecting \(g(w_\alpha^-)\) and \(g(w_\alpha^+)\), such that \(L(\sigma_\alpha)\leq \lambda \, d(g(w_\alpha^-), g(w_\alpha^+))\). 

Now, if \(w=\iota_\alpha(t)\) for some \(t\in I_\alpha\), then we set \(\bar{g}(w)\coloneqq \sigma_\alpha(t)\), and if \(w\in e(Z)\), then we set \(\bar{g}(w)\coloneqq g(w)\). Clearly, \(\bar{g}|_{W_\alpha}\) is \(\lambda\)-Lipschitz. Since \(W\) is an \(\R\)-tree, for every pair of points \(w_\alpha\in W_\alpha\), \(w_\beta\in W_\beta\) with \(\alpha\neq \beta\), there exist \(w^\prime_\alpha \in \partial W_\alpha\), \(w^{\prime}_\beta\in \partial W_\beta\), such that \(d(w_\alpha, w_\beta)=d(w_\alpha, w^\prime_\alpha)+d(w^\prime_\alpha, w^{\prime}_\beta)+d(w^{\prime}_\beta, w_\beta)\). Hence, a short calculation shows that \(\bar{g}\) is \(\lambda\)-Lipschitz, as was to be shown.
\end{proof}

\begin{proof}[Proof of Corollary~\ref{cor:finite-dist-set}]
Due to Corolary~\ref{cor:main2} it suffices to show that \(Z\) has Nagata dimension zero with constant \(2^{n-2}\). In the following, we suppose that \(n\geq 3\). The cases \(n=1\) and \(n=2\) can easily be treated directly. The following construction is due to Fefferman (see \cite[Lemma 3.2]{MR2233850}). Fix \(s>0\). We may suppose that \(s\geq \sep Z\).  Since $0$, $\sep Z\in \{d(z,z') : z,z' \in Z\} \setminus (s,\infty)$, there exists \(\ell_0\in \{0, \dots, n-2\}\) such that 
\[
\big\{ d(z,z') : z,z'\in Z\big\}\cap (2^{\ell_0} s, 2^{\ell_0+1} s]=\varnothing
\]
by the pigeonhole principle.
By construction, it follows that the relation 
\[
R=\big\{ (z,z') : d(z,z')\leq  2^{\ell_0} s\big\}
\]
is an equivalence relation on \(Z\). Clearly, \(\diam\, [z]\leq 2^{\ell_0} s\leq 2^{n-2} s\) and \(d([z], [z'])>2^{\ell_0+1} s> s\) whenever \([z]\neq [z']\). Since \(s\geq \sep Z\) was arbitrary, this implies that \(Z\) has Nagata dimension zero with constant \(2^{n-2}\), as desired.
\end{proof}

\section{An optimal lower bound for finite ultrametric spaces}\label{sec:lower-bound}

The aim of this section is to prove Theorem \ref{thm:lowerBound} from the introduction. Let \((X_n,d)\) denote the leaves of the full binary tree of height $n$. Notice that $X_n$ can be realized as $\{0,1\}^n$ with distance function defined by \[
d((x_1,\dots,x_n),(x_1^\prime,\dots, x_n^\prime)) \coloneqq  2\max\{k : x_k \neq x_k^\prime\}.
\]
Moreover, \(\diam(X_n)=2n\) and if \(B=B(x, 2m)\subset X_n\) for \(m\in \{0,\dots n\}\), then \(B\) is isometric to \(X_{m}\). 
It is easy to check that for every ball \(B=B(x, 2s)\subset X_n\), one has \(B=B(x', 2s)\) for every point \(x'\in B\). Hence, \(X_n\) is an ultrametric space for every \(n\geq 1\). 

Given a tree \(T=(X_n, E)\) we let \(D(X_n\to T)\) denote the smallest real number \(D\geq 0\) such that \(d_T(x,x')\leq D d(x,x')\) for all \(x, x'\in X_n\). Throughout this section, the metric \(d_T=d_\omega\) is the shortest-path metric of the weighted tree \((T, \omega)\) with \(\omega\coloneqq d|_E\); see \eqref{eq:weight-function-1} for the definition of \(d_\omega\). 
We put
$$
K_\infty\coloneqq  \sup_{n\geq 1} \, \,\inf_{\substack{\\[0.088em] T=(X_n, E) \text{ tree}}} \,D(X_n\to T)
$$
and we say that a tree $T=(X_n, E)$ is \textit{$s$-admissible} if $n\geq s$ and $D(X_n\to T) \leq K_\infty$. Recall that the \textit{radius} of a metric space \(X\) is defined as 
\[
\rad\,(X)\coloneqq \inf_{\substack{\\[0.088em]x\in X}} \, \sup_{x'\in X}\,  d(x,x').
\] 
Our main tool in the proof of Theorem \ref{thm:lowerBound} is the following notion:
\begin{definition}\label{def:rad}Let \(\rad\colon (0,\infty)\to [0, \infty)\) be defined by
\begin{align*}
\rad(s)=\inf\bigl\{ \rad\, (B\subset T[X_n])\bigr\},
\end{align*}
where the infimum is taken over all \(s\)-admissible trees \(T=(X_n, E)\) and subsets \(B=B(x,2s) \subset X_n\). Here, \(B\subset T[X_n]\) means that \(B\) is considered as a metric subspace of \(T[X_n]\) equipped with the subspace metric. The metric space \(T[X_n]=(X_n, d_{T})\) is defined in Section~\ref{sec:GraphTheory}.
\end{definition}
By Proposition~\ref{prop:optimalWeights}, 
\begin{equation}\label{eq:reduction-top-1}
K_\infty=\sup_{n\geq 1} \, \,\inf_{\substack{\\[0.088em] Y \text{ weighted tree}}}\, d_{\bl}(X_n, Y);
\end{equation}
thus, to prove Theorem~\ref{thm:lowerBound} it suffices to show that \(K_\infty\geq 8\). On the one hand, by definition,
\begin{equation}\label{eq:upperBound1}
\frac{\rad(s)}{2s}\leq K_\infty
\end{equation}
for every $s>0$, but on the other hand, we will see that if \(K_\infty <8\) then 
\(s\mapsto \rad(s)\cdot (2s)^{-1}\) is unbounded, which contradicts \eqref{eq:upperBound1}. Hence, it seems worthwhile to have tools to obtain lower bounds  on \(\rad(s)\). The following notion will prove very useful for this purpose.

\begin{definition}\label{def:conn}Let \(\Conn(s) \subset [0,1)\) denote the subset consisting of all $\alpha\in [0,1)$, such that for every $s$-admissible tree $T=(X_n, E)$ and every pair of points $o,x\in X_n$ with $x\in \min_{\preceq_{o, T}} S(o, 2\lceil s \rceil)$, one has 
\[
B(x, 2\alpha s) \subset \up_{\preceq_{o, T}}[x].
\]
The partial order \(\preceq_{o, T}\) on \(X_n\) is introduced in Section \ref{sec:PartialOrder} and the sphere \(S(o, 2\lceil s \rceil)\subset X_n\) is defined in \eqref{eq:defBall}.
\end{definition}

\begin{lemma}\label{lem:basecase}
Let \(0\leq \alpha <1/4\,\) be a real number. Then \(\alpha\in\Conn(s)\) for every \(s>0\). 
\end{lemma}

\begin{proof}
For the sake of a contradiction, suppose there exists \(\alpha\in (0, 1/4)\) such that \(\alpha\notin \Conn(s)\) for some \(s>0\). By definition of \(\Conn(s)\), there exist an \(s\)-admissible tree $T=(X_n, E)$, and points $o\in X_n$, $x \in \min_{\preceq_{o, T}} S(o,2\lceil s\rceil)$ and $y \in B(x, 2\alpha s)$, such that $y \notin \up_{\preceq_{o, T}}[x]$. By construction, the path \([x,y]_T\) contains at least two edges of length $\geq 2\lceil s \rceil$, and so \(d_T(x,y) \geq 4\lceil s \rceil\). 
Since $y\in B(x,2\alpha s)$,
$$
K_\infty \geq \frac{d_T(x,y)}{d(x,y)} \geq \frac{4 \lceil s \rceil}{2\alpha s} >8.
$$
But by Theorem~\ref{thm:main1}, \(K_\infty \leq 8\). This is a contradiction and thus $\alpha\in \Conn(s)$, as desired.
\end{proof}

There is the following interplay between \(\Conn(s)\) and \(\rad(s)\):

\begin{lemma}
\label{lem:upperBound:i=0}
Let \(\alpha\in [0,1)\) and \(s_0>0\) be real numbers. If $\alpha\in \Conn(s)$ for all $s\geq s_0$, then for every $\varepsilon>0$, there exists $s_1\geq s_0$ such that
$$
\rad(s) \geq 2s\, \bigg(\frac{1}{1-\alpha}-\varepsilon\bigg)
$$
for all $s\geq s_1$.
\end{lemma}

\begin{proof}
Fix \(s>s_0\) such that \(\lfloor s \rfloor \geq s_0\). Let $T=(X_n, E)$ be an $s$-admissible tree. Let \(o\in X_n\) and let $x\in \min_{\preceq_{o,T}} S(o,2\lfloor s\rfloor)$. Since \(\alpha \in \Conn(\lfloor s\rfloor)\), we have  $B(x,2\alpha  \lfloor s\rfloor )\subset \up_{\preceq_{ o,T}}[x]$. By definition of $\rad$, there exists $\bar{x}\in B(x,2\alpha \lfloor s\rfloor )$ with $d_T(x,\bar{x}) \geq \rad(\alpha  \lfloor s\rfloor)$, and so
\begin{equation*}
d_T(o, \bar{x})=d_T(o, x)+d_T(x, \bar{x}) \geq 2 \lfloor s\rfloor  + \rad(\alpha  \lfloor s \rfloor ).
\end{equation*}
Hence, since 
\[
\rad(s)=\inf\limits_{T=(X_n, E) \text{ \(s\)-admissible} } \hspace{0.75em} \min_{o\in X_n} \hspace{0.5em}\max_{x\in B(o, 2s)} d_T(o, x) 
\]
it follows that
\begin{equation}\label{eq:rad1}
\rad(s) \geq 2 \lfloor s\rfloor  + \rad(\alpha  \lfloor s \rfloor ).
\end{equation}
By using \eqref{eq:rad1} repeatedly, we find by means of the geometric series that if \(s>\alpha^{-N} s_0\) for some \(N\geq 1\) sufficiently large, then
$$
\rad(s) \geq 2s\, \bigg(\frac{1}{1-\alpha}-\varepsilon\bigg),
$$
which is the asserted inequality. 
\end{proof}

Now we are in a position to prove Theorem \ref{thm:lowerBound}.

\begin{proof}[Proof of Theorem \ref{thm:lowerBound}]
For the sake of a contradiction, suppose that $K_\infty<8$. Let $\varepsilon_0>0$ and $r>1$ be real numbers such that
\begin{equation}
\label{eq:upperBound2}
rK_\infty <8-2\varepsilon_0.
\end{equation}

Further, fix $1/5< \alpha_0 <1/4$ and choose a sequence \((\alpha_i)\subset (0,1)\) satisfying $\alpha_{i-1}< \alpha_i\leq r \alpha_{i-1}$ for all $i\geq 1$ and $\alpha_{i} \to 1$ as $i\to \infty$. We will prove the following claim by induction on \(i\):

\begin{claim}\label{claim:1}
For every $i\geq 0$ there exists $s_i>0$ such that for all $s\geq s_i$,
\begin{enumerate}
\item\label{it:claim1} $\alpha_i\in \Conn(s)$,
\item\label{it:claim12} $\rad(s) \geq 2s\left(\frac{1}{1-\alpha_i}-\varepsilon_0\right)$.
\end{enumerate}
\end{claim}

Having Claim~\ref{claim:1} at hand, a contradiction is obtained immediately. Indeed, since $\alpha_i \to 1$ as $i\to \infty$, there exists $i_0$ with $\alpha_{i_0}> 1-\frac{1}{K_{\infty}+\varepsilon_0}$ and due to Claim~\ref{claim:1} for every $s\geq s_{i_0}$,
$$
\frac{\rad(s)}{2s} \geq \biggl(\frac{1}{1-\alpha_{i_0}}-\varepsilon_0\biggr)>K_\infty,
$$
which contradicts \eqref{eq:upperBound1} and so the case when \(K_\infty <8\) is not possible. Hence, by \eqref{eq:reduction-top-1}, for every \(K\in [1,8)\) there exists an integer \(n(K)\geq 1\) with the desired property. To finish the proof, we need to establish Claim \ref{claim:1}.

\textit{Proof of Claim \ref{claim:1}}: 
By Lemma~\ref{lem:basecase}, it follows that \(\alpha_0\in \Conn(s)\) for all \(s>0\). In particular, this is true for all \(s\geq 1\). Hence, by Lemma~\ref{lem:upperBound:i=0}, there exists \(s_0 \geq 1\), such that $\rad(s) \geq 2s\left(\frac{1}{1-\alpha_0}-\varepsilon_0\right)$ for all \(s\geq s_0\). This proves the claim when \(i=0\). So let us suppose that \(i\geq 1\) and Claim \ref{claim:1} holds for $i-1$. We proceed by showing \eqref{it:claim1}. We claim that $\alpha_i\in \Conn(s)$ for every $s\geq \eta s_{i-1}$, where 
\[
\eta\coloneqq \frac{1}{(1-\alpha_i)\alpha_{i-1}^2}.
\] 
Suppose that this is not the case for some \(s\geq \eta s_{i-1}\). Thus, there exist an $s$-admissible tree $T=(X_n, E)$, and points $o\in X_n$, $x\in \min_{\preceq_{o, T}}S(o,2\lceil s \rceil)$ and $y \in B(x, 2\alpha_i s)$, such that $y \notin \up_{\preceq_{o, T}}[x]$. As \(s\geq \frac{1}{1-\alpha_i}\), we have \(y\in S(o,2\lceil s \rceil)\). We may suppose that
\begin{equation}\label{eq:besty}
y =  \min_{\preceq_{o, T}} \big([o,y]_T \cap  B(y, 2\alpha_i s)\big).
\end{equation}
Now, there are two cases: either \(y\in \min_{\preceq_{o, T}}S(o,2\lceil s \rceil)\) or the maximum
\begin{equation}\label{eq:bestp}
p\coloneqq \max_{\preceq_{o, T}} \big([0,y]_T \cap S(o, 2\lceil s \rceil)\setminus y \big) 
\end{equation}
exists.
To begin, we consider the case when the point \(p\) defined in \eqref{eq:bestp} exists. 

We abbreviate \(2\beta s\coloneqq d(p,y)\). 
Using \eqref{eq:besty} and \eqref{eq:bestp}, we find $ 2\alpha_i s < 2\beta s <   2 \lceil s \rceil$ and \(y\in \min_{\preceq_{p, T}}S(p,2 \beta s)\). Since \(\beta \alpha_{i-1}s >s_{i-1}\) and Claim \ref{claim:1} holds for \(i-1\), we have \(B(y,2\alpha_{i-1}(\beta s))\subset \up_{\preceq_{p, T}}[y]=\up_{\preceq_{o, T}}[y]\), so there exists $\bar{y}\in B(y,2\beta \alpha_{i-1}s) \cap \up_{\preceq_{o, T}}[y]$ such that
$$
d_T(y, \bar{y}) \geq \rad(\beta \alpha_{i-1}s)\geq 2\beta \alpha_{i-1} s \,\biggl(\frac{1}{1-\alpha_{i-1}}-\varepsilon_0\biggr).
$$
Moreover, since $\alpha_{i-1}s> s_{i-1}$ and  $x\in \min_{\preceq_{o, T}}S(o,2\lceil s \rceil)$, we have $B(x,2\alpha_{i-1} s)\subset \up_{\preceq_{o, T}}[x] $, so there exists $\bar{x}\in \up_{\preceq_{o, T}}[x] \cap B(x,2\alpha_{i-1}s)$ such that
$$
d_T(x, \bar{x})\geq \rad(\alpha_{i-1} s)\geq 2\alpha_{i-1}s\,\biggl( \frac{1}{1-\alpha_{i-1}}-\varepsilon_0\biggr).
$$
Using that the path \([x,p]_T\) contains at least two edges of length $\geq 2\lceil s \rceil$, we find 
\begin{align*}
d_T(\bar{x}, \bar{y}) &\geq d_T(\bar{x},x) + d_T(x,p) + d_T(p,y) + d_T(y, \bar{y}) \\
&\geq \rad(\alpha_{i-1} s) +4\lceil s \rceil+2\beta s+\rad(\beta \alpha_{i-1}s).
\end{align*}
Since \(d(\bar{x},\bar{y}) \leq 2\alpha_i s\), we arrive at
\begin{align*}
K_\infty &\geq \frac{d_T(\bar{x},\bar{y})}{d(\bar{x},\bar{y})} \geq \frac{\alpha_{i-1}}{\alpha_i}\bigg( \frac{2}{\alpha_{i-1}}+\frac{\rad(\alpha_{i-1}s)}{2\alpha_{i-1}s}+\beta\Big(\frac{1}{\alpha_{i-1}}+\frac{\rad(\beta \alpha_{i-1}s)}{2\beta \alpha_{i-1}s} \Big)\bigg) 
\end{align*}
and so, as \(\beta >\alpha_{i-1}\), we obtain
\begin{align*}
K_\infty &\geq \frac{1}{r}\bigg( \frac{2}{\alpha_{i-1}}+\frac{1}{1-\alpha_{i-1}}+\alpha_{i-1}\Big(\frac{1}{\alpha_{i-1}}+\frac{1}{1-\alpha_{i-1}} \Big)\bigg)-\frac{\varepsilon_0}{r}(1+\beta).
\end{align*}
Therefore,
\[K_\infty \geq \frac{8}{r}-\frac{2\varepsilon_0}{r}>K_\infty,\]
where we have used \eqref{eq:upperBound2} in the last inequality. This is a contradiction.

It remains to consider the case where $y\in \min_{\preceq_{o, T}}S(o,2\lceil s \rceil)$. Since \(s>s_{i-1}\), there exist $\bar{x}\in B(x,2\alpha_{i-1}s)\subset  \up_{\preceq_{o, T}}[x]$ and $\bar{y}\in B(y,2\alpha_{i-1}s)\subset  \up_{\preceq_{o, T}}[y]$, such that \(d_T(x, \bar{x})\geq \rad(\alpha_{i-1} s)\) and \(d_T(y,\bar{y}) \geq \rad(\alpha_{i-1} s)\).
Consequently, 
\begin{align}\label{eq:estimate1}
K_\infty &\geq \frac{d_T(\bar{x},\bar{y})}{d(\bar{x},\bar{y})}\geq \frac{d_T(\bar{x},x) + d_T(x,y) + d_T(y, \bar{y})}{2\alpha_i s} \nonumber\\
&\geq \frac{\alpha_{i-1}}{\alpha_i}\Big(\frac{\rad(\alpha_{i-1} s)}{2\alpha_{i-1} s}+\frac{2}{\alpha_{i-1}}+\frac{\rad(\alpha_{i-1} s)}{2\alpha_{i-1} s}\Big).
\end{align}
As \(\alpha_{i-1}s>s_{i-1}\), \eqref{eq:estimate1} yields
\[
K_\infty \geq \frac{1}{r}\Big( \frac{2}{1-\alpha_{i-1}}+\frac{2}{\alpha_{i-1}}\Big)-\frac{2\varepsilon_0}{r}\geq \frac{8}{r}-\frac{2\varepsilon_0}{r}>K_\infty.
\]
This is a contradiction and so we have shown that $\alpha_i\in \Conn (s)$ whenever \(s\geq \eta s_{i-1}\). Now, by Lemma \ref{lem:upperBound:i=0}, there exists $s_i \geq \eta s_{i-1}$ such that for all $s\geq s_i$,
$$
\rad(s) \geq 2s\, \Bigl(\frac{1}{1-\alpha_i}-\varepsilon_0\Bigr),
$$
establishing Claim~\ref{claim:1}\eqref{it:claim12}. This completes the proof of Claim~\ref{claim:1}.
\end{proof}

\begin{remark}\label{rem:exp-bound}
Let \(K\in [3,8)\). In the following, we briefly outline how to use our proof of Theorem~\ref{thm:lowerBound} to find an explicit upper bound on \(n(K)\) which is exponential in \(K\). By setting \(r\coloneqq 1+(8-K)/2K\) and \(\varepsilon_0\coloneqq 1-K/8\), we have \(r K< 8-2\varepsilon_0\). If \(\alpha_0=2/9\) and \(\alpha_i=r \alpha_{i-1}\), then a short calculation reveals that
\(\alpha_{i_0}>1-1/(K+\varepsilon_0)\), where \(i_0\coloneqq \lceil 21/(8-K) \rceil \). In Lemma~\ref{lem:upperBound:i=0} one can take \(s_1= \bigl(16/\varepsilon\bigr)^{13} \max\{s_0, 144/(7\varepsilon)\}\) for \(\alpha \leq 7/8\) and \(0< \varepsilon < 5/8\), where $\max\{s_0, 144/(7\varepsilon)\}$ is introduced to keep the rounding errors below a factor of $\varepsilon/2$. As \(\eta \leq 28\), by considering the proof of Claim~\ref{claim:1}, one can show that
\[
s_{i_0} \leq \frac{144}{7\varepsilon_0} \Big( 28 \,\Big(\frac{16}{\varepsilon_0}\Big)^{13} \Big)^{i_0}.
\]
This yields the following exponential upper bound on \(n(K)\):
\begin{equation}\label{eq:exponential-upper-bound}
n(K) \leq \exp\Big(\frac{8000}{(8-K)^2}\Big), 
\end{equation}
Notice that \eqref{eq:exponential-upper-bound} clearly also holds when \(K\in [1,3)\).
\end{remark}

\begin{remark}\label{rem:generalization-to-c-bigger-than-zero}
As pointed out in Section~\ref{sec:main-result-and-examples}, a metric space \(X\) has Nagata dimension zero with constant \(c\) if \(\diam\, [x]_s \leq cs\) for all \(x\in X\) and all \(s>0\). 
Instead of considering closed balls, \(\rad(s)\) and \(\Conn(s)\) can also be defined in terms of the equivalence classes \([x]_{2s}\) and \([x]_{2\alpha s}\), respectively. Indeed, if \(X\) is an ultrametric space, then \([x]_s=B(x,s)\) for all \(x\in X\) and all \(s>0\). 
Hence, if we assume the existence of a suitable family \(\{X_n\}\) of finite metric spaces with \(c_N(X_n)=c_\ast>1\), then, using the new definitions of \(\rad(s)\) and \(\Conn(s)\) in terms of \([x]_s\), our proof of Theorem~\ref{thm:lowerBound} can be adapted to show that Theorem~\ref{thm:main1} is sharp when \(c=c_\ast\). However, we do not have a good candidate for such a family of metric spaces. 
\end{remark}

\section{Approximating subsets of \(\R\)-trees by weighted trees}
In this section, we prove Theorem \ref{thm:guptaGen}. Before proceeding with its proof, let us start with the proof of Corollary~\ref{cor:matousek}. 
\begin{proof}[Proof of Corollary~\ref{cor:matousek}]
Using Theorem~\ref{thm:guptaGen}, we obtain an \(8\)-bilipschitz embedding \(e\colon Z\to W\) into a complete \(\R\)-tree \(W\) such that \(\sing W \subset e(Z)\). Now, the proof is exactly parallel to that of Corollary~\ref{cor:main2}. 
\end{proof} 
As discussed in the introduction, the proof of Theorem~\ref{thm:guptaGen} is done in three steps. As a first step, we adapt a construction due to Gupta (see \cite{MR1958411}) to discrete proper 0-hyperbolic spaces.
Recall that a metric space \(X\) is called \textit{discrete} if for every \(x\in X\) there exists \(\varepsilon_x>0\) such that \(U(x, \varepsilon_x)=\{x\}\). 

\begin{proposition}\label{prop:Gupta-for-extreme-points}
Let $X$ be a discrete, proper and 0-hyperbolic metric space such that there are no pairwise distinct points $x,y,z\in X$ satisfying $d(x,y)+d(y,z) = d(x,z)$. Then there exists a weighted tree $T=(X,E, \omega)$ such that
\begin{equation}\label{eq:the-result}
    d(x,x') \leq d_T(x,x') < 8d(x,x')
\end{equation}
for all $x,x'\in  X$.
\end{proposition}
 
\begin{proof}
Let $E(X)$ denote the injective hull of \(X\), which is an $\R$-tree by Proposition~\ref{prop:dress-1}. For convenience, we suppose that $X\subset E(X)$ is already isometrically embedded. By Lemma~\ref{lem:convex-hull}, the subset \(W\subset E(X)\) defined by
\[
W\coloneqq\bigcup_{x, x'\in X} [x, x']_{E(X)}
\]
is convex, and thus as it is also \(0\)-hyperbolic, it is an \(\R\)-tree. By our assumption on \(X\), if \(x\), \(y\), \(z\in X\) are such that \(z\in [x, y]_W\), then \(z=x\) or \(z=y\). 
Fix a point $o\in W\setminus X$ and recall the partial order $\preceq_{o,W}$ from Section~\ref{sec:PartialOrder}. For later use, we note that it follows immediately from the definition of \(W\) that the set \(X\cap \up_{\preceq_{o,W}}[w]\) is non-empty for any \(w\in W\). As $X$ is discrete, \(d(o, X)>0\). Set $V_0\coloneqq \{o\}$ and choose a point $c(o)\in X\cap \up_{\preceq_{o,W}}[o]$ such that $d(o, c(o))= d(o,X\cap \up_{\preceq_{o,W}}[o])$. 

Suppose that $V_i\subset W\setminus X$ is given and for every $v\in V_i$ we have selected a point $c(v)\in X\cap \up_{\preceq_{o,W}}[v]$ at minimal distance to \(v\), that is, $d(v,c(v))= d(v,X\cap \up_{\preceq_{o,W}}[v]).$ Then, we define for every $v\in V_i$ the set $S(v) \subset W\setminus X$ by
$$
S(v)\coloneqq  \up_{\preceq_{o,W} }[v]\cap S_{\hspace{-0.15em}W}\Big(v, \frac{d\big(v,c(v)\big)}{2}\Big).
$$
For every $w\in S(v)$ we put $\pred(w)\coloneqq  v$. For the unique $w\in S(v)$ with $w\in [v,c(v)]_W$, we define $c(w)\coloneqq  c(v)$. For all the other $w'\in S(v)$, choose $c(w') \in \up_{\preceq_{o,W} }[w']\cap X$ at minimal distance to \(w'\). This is possible since \(X\) is proper and discrete, and thus the infimum in the definition of \(d(w', X\cap \up_{\preceq_{o,W} }[w'])\) is attained. Set $V_{i+1}\coloneqq \bigcup_{v \in V_i}S(v)$ and define $V\coloneqq  \bigcup_{i=0}^\infty V_i$ and observe that the choices $c(v)\in X$ define a map $c:V\to X$.

We claim that \(c(V)=X\). Fix \(x\in X\) and set \(\delta\coloneqq  d\big(x, X\setminus \{x\}\big)\). As \(X\) is discrete, \(\delta>0\). If $d(x,o)\leq \delta/4$, then $c(o) = x$ and $x\in c(V)$. Otherwise,
it follows from the assumptions that the distance function to $X$ restricted to $[o, x]_W\setminus U_W(x,\delta/4)$ is bounded from below by a positive number \(\varepsilon>0\). Consequently, by setting \(N\coloneqq 2 \lceil \tfrac{d(o, x)}{\varepsilon}\rceil\) we find that there exists \(v\in V_{N}\cap U_W(x,\delta/4)\). By definition of \(\delta\), it follows that \(c(v)=x\). Hence, as \(x\) was arbitrary, \(c(V)=X\), as desired. 

Let $T=(X,E)$ be the graph where $\{x,x'\}\in E$ if and only if $x\neq x'$ and there exist $v\in V$ and $w\in S(v)$ such that $x=c(v)$ and $x'=c(w)$ or vice versa.

In the following, we show by induction that \(T\) is a tree. For every \(i\geq 0\), we put \(X_i\coloneqq \bigcup_{j=0}^i c(V_j)\) and let \(T_i\subset T\) denote the subgraph induced by \(X_i\). Notice that if \(v\), \(v'\in V_{i}\) are distinct, then \(c(v)\neq c(v')\). Consequently, for every \(x\in X_{i+1}\setminus X_i\) there is a unique \(v\in V_{i+1}\) such that \(x=c(v)\) and every point \(v'\in V\) with \(c(v')=x\) is contained in the geodesic \([v,x]_W\). Hence, given \(x\in X_{i+1}\setminus X_i\) there is precisely one point \(x'\in X_{i}\) such that \(x\) and \(x'\) are adjacent in \(T\). Moreover, no two points \(x\), \(x'\in X_{i+1}\setminus X_i\) are adjacent in \(T\). Hence, since \(T_0\) is a tree and \(T_{i}\subset T_{i+1}\), it follows by induction that \(T_i\) is a tree for every \(i\geq 0\). Since \(T=\bigcup_{i=0}^\infty T_i\) and \(T_i\subset T_j\) for all \(i\leq j\), we have shown that \(T\) is a tree, as claimed. 

Now, let \(\omega\colon E\to \R\) be defined by \(\omega(\{x,x'\})\coloneqq d(x,x')\) and denote by \(d_T\) the shortest-path metric of the weighted tree \(T=(X,E, \omega)\). By the triangle inequality, \(d_T(x,x') \geq d(x,x').\)
We proceed by showing the right inequality of \eqref{eq:the-result}. We fix $x,x'\in X$ and put
\(x_1\coloneqq  \min_{\preceq_{c(o),T}}\bigl([x,x']_T\bigr).\)
Let \((x_1, \dots, x_m)\) and \((x_1^\prime, \dots, x_n^\prime)\) denote the paths \([x_1, x]_T\) and  \([x_1, x']_T\), respectively. For all \(1\leq k \leq m-1\), we define 
\begin{align*}
p_k\coloneqq  &\min_{\preceq_{o,W}} \,[x_k,x_m]_W,  &v_k &\coloneqq  \min_{\preceq_{o,W}} \,(V\cap [p_k,p_{k+1}]_W\setminus\{p_k\}).
\end{align*}
 See Figure~\ref{fig:one} for an illustration.  Further for $k\geq 2$, we put
\begin{align*}
    a_k&\coloneqq  d(v_{k-1}, p_k),& b_k&\coloneqq  d(p_k, v_k),&  c_k&\coloneqq  d(p_k, x_k).
\end{align*}
Notice that \(v_k\in [x,v_1]_W\),  \(c(v_k)=x_{k+1}\),  and $ \pred(v_1)=\pred(v_1')\eqqcolon v_0$.
\begin{figure}[t]
    \centering
    \includegraphics[width=\textwidth]{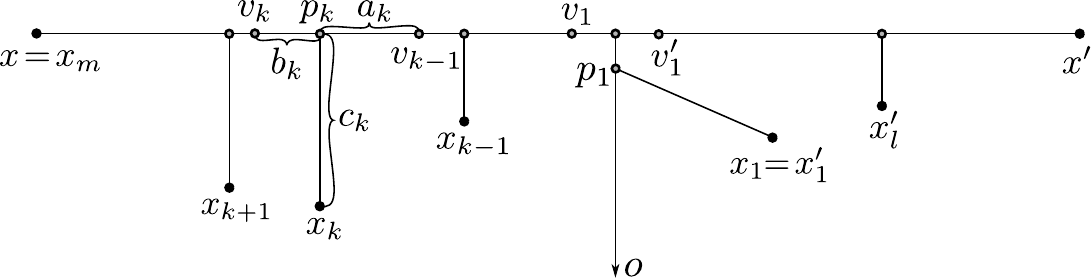}
    \caption{Illustration of the points \(x_k\), \(p_k\) and \(v_k\), and the distances \(a_k\), \(b_k\) and \(c_k\).}
    \label{fig:one}
\end{figure}
By construction,
\begin{align}\label{eq:aux3}
    d_T(x_m,x_1)= d(x_m,v_1) + 2 \sum_{k=2}^{m-1} c_k 
+ d(x_1,v_1).
\end{align}
Moreover, we have $\frac{d(v_{k-1}, x_k)}{2} = \frac{a_k + c_k}{2} \leq a_k+ b_k$ and thus $c_k\leq a_k + 2 b_k$ with equality if and only if $v_{k-1} = \pred(v_k)$, by construction of \(T\). By plugging this into \eqref{eq:aux3} we find
\[
d_T(x_m,x_1) \leq d(x_m,v_1) + d(x_1,v_1)+ 2 c_{m-1} + 4 \sum_{k=2}^{m-2} (a_k + b_k).
\]
Since
\[
d(x_m,v_1)= d(x_m, p_{m-1})+ \sum_{k=2}^{m-2}(a_k+ b_k), 
\]
after rearranging the terms, we get
\(
d_T(x_m,x_1) \leq 5d(x_m,v_1) + \text{I} + \text{II},
\)
where
\[
\text{I}\coloneqq  2c_{m-1}-4 d(p_{m-1},x_m) \quad\quad \text{ and } \quad\quad \text{II}\coloneqq  d(x_1,v_1).
\]
As \(c(v_{m-2})=x_{m-1}\) and $p_{m-1}\in [v_{m-2},x_{m-1}]_W\cap [v_{m-2},x_{m}]_W$, it follows that \(d(p_{m-1}, x_m) \geq  d(p_{m-1},x_{m-1})\). Consequently, $-4 d(p_{m-1},x_m) + 2c_{m-1} <0$, and so we obtain $\text{I}<0$.

Next, we show that \(\text{II}< 3 d(x_m, v_1)\). 
Observe that $d(v_0, x_1) \leq d(v_0,x_m)$, which directly implies $d(p_1,x_1) \leq d(p_1,x_m)$. Since $v_0 = \pred(v_1)$ and $x_1 = c(v_0)$, it follows that $d(p_1,v_1) = d(v_0,x_1)/2 - d(v_0,p_1) \leq d(p_1,x_1)/2$, and so $d(p_1,v_1) \leq d(v_1,x_m)$.  By combining these estimates, we infer
\[
d(x_1,v_1)= d(x_1,p_1) + d(p_1,v_1) \leq \frac{3}{2} d(x_1,p_1) \leq \frac{3}{2} d(p_1,x_m) \leq 3 d(v_1,x_m),
\]
where for the last inequality we have used that \(d(p_1, x_m)=d(p_1, v_1)+d(v_1,x_m)\). By the above, \(d_T(x_m,x_1) < 8d(x_m,v_1)\). Analogously, one can show that \(d_T(x_n',x_1') < 8d(x_n',v_1')\).
Since \(d(x_m,v_1)+d(x_n',v_1')\leq d(x_m, x_n')\), we arrive at \(d_T(x_m,x_n')<8 d(x_m, x_n')\). This proves the proposition.
\end{proof}

Having Proposition~\ref{prop:Gupta-for-extreme-points} at hand, the following lemma tells us that Theorem~\ref{thm:guptaGen} is valid if \(\sep X>0\). Recall that \(\sep X\coloneqq \inf \big\{ d(x,x') : x, x'\in X,\, x\neq x'\big\}\).

\begin{lemma}\label{lem:gluing-sep-bigger-than-zero}
Let \(W\) be a complete \(\R\)-tree and \(X\subset W\) a countable subset with \(\sep X>0\). Let \(C>0\) and suppose that for every connected component \(U\subset W\setminus X\) there is a \(C\)-bilipschitz map \(f_U\colon \partial U \to Y_U\) into an \(\R\)-tree \(Y_U\) such that \(\sing Y_U\subset f_U(\partial U)\). Then there exist an \(\R\)-tree \(Y\) and a \(C\)-bilipschitz map \(f\colon X\to Y\) such that \(\sing Y\subset f(X)\).
\end{lemma}

\begin{proof}
By rescaling (if necessary) we may suppose that for every connected component \(U\subset X\setminus W\),
\begin{equation}\label{eq:rescaled}
d(x,y)\leq d_{Y_U}(f_U(x),f_U(y)) \leq C d(x,y)
\end{equation}
for all \(x\), \(y\in U\). Since \(X\subset W\) is countable, we can write \(W\setminus X=\bigcup_{i=0}^\infty U_i\) for the decomposition of \(W\setminus X\) into open connected components.

We put \(X_i\coloneqq \partial U_i\) and  \(X_{\leq i}\coloneqq \bigcup_{j\leq i} X_j\), and \(W_{i}\coloneqq \overline{U_i}\) and  \(W_{\leq i}\coloneqq \bigcup_{j\leq i} W_j\) for all \(i\geq 0\). Since \(\sep X>0\), for every \(x\), \(x'\in X\) the geodesic \([x,x']_W\) meets only finitely many of the connected components \(U_i\). Thus, by fixing a component  \(U_\varnothing \subset W\setminus X\) the connected components of \(W\setminus X\) can be indexed by finite strings over a countable alphabet.
Consequently, we may choose an enumeration \(\{U_i\}_{i\geq 0}\) of the connected components of \(W\setminus X\) such that  \(U_0=U_\varnothing\) and \(W_i\cap W_{\leq i-1}\neq \varnothing\) for all \(i\geq 1\).  
We put \(\{x_i\}\coloneqq W_i\cap W_{\leq i-1}\) for all \(i\geq 1\). Notice that \(x_i\in X_i\cap X_{\leq i-1 }\).

We define the \(\R\)-trees \(Y_{i}\) and the \(C\)-bilipschitz maps \(f_{i}\colon X_{\leq i}\to Y_{i}\) by induction as follows. 
We set \(Y_{0}\coloneqq Y_{U_0}\) and \(f_{0}\coloneqq f_{Y_0}\). Now, the \(\R\)-tree \(Y_{i}\) is obtained by gluing the \(\R\)-trees \(Y_{U_{i}}\) and \(Y_{i-1}\)  along the points \(f_{U_{i}}(x_{i})\in Y_{U_i}\)
and \(f_{i-1}(x_{i})\in Y_{i-1}\). This is possible since \(x_i\in X_{\leq i-1}\). Further, the \(C\)-bilipschitz map \(f_{i}\colon X_{\leq i}\to Y_{i}\) is defined by \(f_{i}(x)\coloneqq f_{i-1}(x)\) if \(x\in X_{\leq i-1}\) and \(f_{i}(x)\coloneqq f_{U_i}(x)\) if \(x\in X_{i}\). Here, we use \eqref{eq:rescaled} to ensure that \(f_i\) is \(C\)-bilipschitz.  
By construction, \(Y_i\subset Y_j\) for all \(i\leq j\). Hence, the union of metric spaces \(Y=\bigcup_{i=0}^\infty Y_i\) is a well-defined metric space. 

Clearly, \(Y\) is an \(\R\)-tree and \(f\colon X\to Y\) defined by \(f(x)\coloneqq f_i(x)\) if \(x\in X_{\leq i}\) is well-defined and \(C\)-bilipschitz. 
Moreover, for every \(i\geq 0\) there exists a canonical isometric embedding \(\alpha_i\colon Y_{U_i}\to Y\) such that \(\alpha_i(Y_{U_i})\cap \alpha_j(Y_{U_j})\subset f(X)\) if \(i\neq j\). Fix \(y\in Y\setminus f(X)\). There is a unique integer \(i\geq 0\) such that \(y\in \alpha_i(Y_{U_i})\). Since \(\sing Y_{U_i}\subset f_{U_i}(X_{i})\), it is easy to check that \(y\) is a regular point of \(Y\). This implies that \(\sing Y\subset f(X)\), as was to be shown.
\end{proof}

Finally, we give the proof of Theorem~\ref{thm:guptaGen}.

\begin{proof}[Proof of Theorem~\ref{thm:guptaGen}]
Let $X_n\subset X$, \(n\geq 1\), denote a $1/n$-net and $U_{n,\alpha} \subset E(X_n)\setminus X_n$ a connected component. The metric space $X_{n,\alpha}\coloneqq \partial U_{n,\alpha} \subset X_n$ is 0-hyperbolic and \(\sep X_{n,\alpha}\geq 1/n\). By Proposition~\ref{prop:dress-1}, \(E(X_n)\) is a complete \(\R\)-tree and so it follows that there are no three distinct points $x$,$y$,$z\in X_{n,\alpha}$ such that $d(x,z) = d(x,y) + d(y,z).$ 
Hence, by Proposition~\ref{prop:Gupta-for-extreme-points}, there exist an \(\R\)-tree \(Y_{n,\alpha}\) and an \(8\)-bilipschitz map \(f_{n,\alpha}\colon X_{n,\alpha} \to Y_{n,\alpha}\) which satisfies \(\sing  Y_{n,\alpha}\subset f_{n,\alpha}\big(X_{n,\alpha}\big)\).

Now, on account of Lemma~\ref{lem:gluing-sep-bigger-than-zero} for each \(n\geq 1\) there is an \(8\)-bilipschitz map \(f_n\colon X_n\to Y_n\) into an \(\R\)-tree \(Y_n\) such that \(\sing Y_n \subset f_n(X_n)\). 
Let $\mathfrak{U}$ be a free ultrafilter on $\N$. 
Clearly, $f_{\mathfrak{U}}: (X_n)_{\mathfrak{U}} \to (Y_n)_{\mathfrak{U}}$ defined by \(f_{\mathfrak{U}}((x_n)_\mathfrak{U})\coloneqq (f_n(x_n))_\mathfrak{U}\) is \(8\)-bilipschitz. By virtue of Lemma~\ref{lem:ultrafilter}, $Y\coloneqq (Y_n)_{\mathfrak{U}}$ is a complete \(\R\)-tree and 
\begin{equation}\label{eq:inclusion-sing}
\sing Y \subset \bigl(\sing Y_n\bigr)_{\mathfrak{U}} \subset  \big(f_n(X_n) \big)_{\mathfrak{U}}=f_{\mathfrak{U}}\big((X_n)_{\mathfrak{U}}\big). 
\end{equation}
Using that \(X\) is proper, we find that \((X_n)_{\mathfrak{U}}\) and \(X\) are isometric via the map \(\iota\colon (X_n)_{\mathfrak{U}}\to X\) defined by \(\iota((x_n)_\mathfrak{U})\coloneqq \lim_\mathfrak{U} x_i\). Hence, \(f\coloneqq f_{\mathfrak{U}} \circ \iota^{-1}\) is \(8\)-bilipschitz and because of \eqref{eq:inclusion-sing}, \(\sing Y\subset f(X)\). This completes the proof. 
\end{proof}

\printbibliography

\end{document}